\documentclass[hidelinks,onefignum,onetabnum]{siamart251216}

\usepackage{amssymb}
\usepackage{bm}
\usepackage{graphicx}
\usepackage{algorithmic}
\usepackage{booktabs}
\usepackage{placeins}
\newsiamremark{assumption}{Assumption}
\newsiamremark{remark}{Remark}

\providecommand{\Curl}{\operatorname{Curl}}
\providecommand{\sCurl}{\operatorname{sCurl}}
\providecommand{\Div}{\operatorname{div}}
\providecommand{\DivDiv}{\operatorname{divdiv}}

\headers{Boundary-Corrected Mixed FEM for the Babu\v{s}ka Paradox}
{P. Tian, S. Wu, and H. Zhou}
\title{A Mixed Finite Element Method for the Babu\v{s}ka Paradox Using Only Discrete Geometry\thanks{{\bfseries Funding:} National Natural
Science Foundation of China under Grant Nos.~12571383 and 12288101.}}
\author{Pengjie Tian\thanks{School of Mathematical Sciences, Peking University,
Beijing 100871, China (\email{pengjie.tian@math.pku.edu.cn}, \email{snwu@math.pku.edu.cn}, and
\email{zhouhao23@pku.edu.cn}).}
\and Shuonan Wu\footnotemark[2]
\and Hao Zhou\footnotemark[2]}

\ifpdf
\hypersetup{
  pdftitle={A Mixed Finite Element Method for the Babu\v{s}ka Paradox Using Only Discrete Geometry},
  pdfauthor={Pengjie Tian, Shuonan Wu, and Hao Zhou}
}
\fi

\begin{document}

\maketitle

\begin{abstract}
The classical Babu\v{s}ka paradox shows that solutions of simply supported
plate problems on polygonal approximations of a curved domain may converge
to an unintended limit. We develop a boundary-corrected
\(H(\DivDiv;\mathbb S)\)--\(L^2\) mixed finite
element method for the simply supported Kirchhoff--Love plate problem. 
The correction uses only the discrete boundary geometry.
Introducing the bending moment as an independent unknown allows the
condition \(M_{nn}=0\) to be imposed directly, while an
edgewise mean constraint on the effective shear suppresses the leading
geometric inconsistency. This constraint improves the boundary consistency
error from \(\mathcal O(h^{1/2})\) to \(\mathcal O(h^{3/2})\). The associated
boundary corrections act in the kernel of \(\DivDiv\), leaving the discrete
equilibrium equation unchanged and yielding a uniformly stable scheme.
Under suitable regularity assumptions, we prove \(L^2\)-error estimates
of order \(h^{3/2}\)
for the bending moment and the broken Hessian of the postprocessed
displacement, and of order \(h^2\) for the displacement. The analysis covers
multiply connected domains and polygonal approximations whose boundaries may
cross the physical boundary. Numerical experiments confirm these rates and
the improvement over the uncorrected method.
\end{abstract}



\section{Introduction}
\label{sec:introduction}

The Babu\v{s}ka paradox is a classical failure of geometric approximation
for the simply supported Kirchhoff--Love plate. The issue goes back to
Babu\v{s}ka's study of domain perturbations \cite{Babuska1963} and was
systematically analyzed by Babu\v{s}ka and Pitk\"aranta
\cite{BabuskaPitkaranta1990}. Even when a sequence of polygonal domains
converges to a smooth domain and each polygonal problem is solved exactly,
the corresponding solutions may converge to a different limit \cite{MazyaNazarov1986,NazarovSweersStilyanou2011,
ChechkinLukkassenMeidell2008}.

The paradox stems from the dependence of the natural plate boundary
condition on boundary curvature. Let \(\Omega\subset\mathbb R^2\) be the
plate midsurface, \(\Gamma:=\partial\Omega\), and \(f \in L^2(\Omega)\) the transverse
load. After scaling the flexural rigidity, the energy of an isotropic
plate with Poisson ratio \(0\leq\sigma<1\) is
\begin{equation*}
  E_\Omega(v):=\frac12\int_\Omega
  \bigl((1-\sigma)|D^2v|^2+\sigma(\Delta v)^2\bigr)\,\mathrm dx
  -\int_\Omega fv\,\mathrm dx,
  \quad v\in H^2(\Omega)\cap H_0^1(\Omega).
\end{equation*}
Let \(u\) be its minimizer and set
\(\mathbb C_\sigma\bm N:=(1-\sigma)\bm N
+\sigma\operatorname{tr}(\bm N)\bm I\). In terms of the bending moment,
the Euler--Lagrange problem reads
\begin{equation}
\label{eq:intro-strong-problem}
  \bm M=\mathbb C_\sigma D^2u,\quad
  \DivDiv\bm M=f\quad\text{in }\Omega,\qquad
  u=0,\quad M_{nn}=0\quad\text{on }\Gamma.
\end{equation}
Here, \(\bm I\) is the identity matrix and
\(M_{nn}:=\bm n^\top\bm M\bm n\), with \(\bm n\) the outward unit normal.
The curvature dependence becomes explicit when \(M_{nn}\) is expressed in
terms of \(u\). On a smooth boundary piece, differentiating the boundary
identity \(u|_\Gamma=0\) twice along the boundary gives
\(M_{nn}=\Delta u-(1-\sigma)\kappa\partial_n u\), where \(\kappa\) is the
signed curvature. On an open edge of a polygon, \(\kappa=0\) and the
condition reduces to \(\Delta u=0\). Thus, polygonal approximation changes
the natural boundary condition and may lead to a different limiting problem
\cite{Davini2002,Davini2003,DeCosterNicaiseSweers2019}.

Two distinct strategies have been used to recover the correct limit. One
approximates the physical boundary to higher order using curved or
isoparametric elements \cite{BrennerNeilanSung2013,ArnoldWalker2020}.
This replaces the classical polygonal setting by a higher-order geometric
approximation. The other retains polygonal domains and modifies the
discrete boundary conditions. In primal methods, this is achieved by
relaxing the edgewise displacement constraint, for example through vertex
conditions or modified and penalized boundary conditions
\cite{Scott1977,RannacherParadox1979,UtkuCarey1983}. Related nonconforming
and mixed approaches were studied in
\cite{RannacherMixed1979,DaviniPitacco2000}. More recently, Bartels and
Tscherner \cite{BartelsTscherner2025} characterized the underlying
principle through compatibility between the discrete boundary constraints
and the approximation of admissible functions. 

We retain the same polygonal setting and develop a {\it boundary-corrected
mixed method} using only the discrete boundary geometry. The Green identity
identifies this correction as a natural mixed counterpart of the primal
relaxation. The displacement trace is paired with the effective shear
\(q_n(\bm M)\). Enlarging the admissible primal trace space therefore
corresponds to restricting its conjugate shear modes. This viewpoint is
particularly natural in \(H(\DivDiv;\mathbb S)\), where the bending moment
is an independent unknown and its generalized boundary trace comprises the
normal--normal moment, effective shear, and corner forces. F\"uhrer and
Heuer \cite{FuehrerHeuer2025} constructed a two-dimensional conforming
element with these degrees of freedom and established its second-order
approximation properties on polygonal meshes. For other conforming
\(H(\DivDiv;\mathbb S)\) elements, see
\cite{ChenHuangDivDiv2020,HuMaZhang2021,ChenHuang3D2022,
ChenHuangArbitrary2022,ChenHuangHybrid2025,HuLiangMaZhang2024}.
For a simply supported plate, \(M_{nn}=0\) is imposed directly on the
moment, whereas the displacement condition is encoded variationally.
The standard mixed discretization therefore avoids the wrong-limit
behavior of the classical paradox.

Avoiding the wrong limit does not, however, prevent geometry-induced
order reduction. On a polygonal approximation, an extension of the exact
displacement has an \(\mathcal O(h^2)\) trace on the straight boundary
edges, while the inverse trace estimate for the discrete effective shear
carries a factor \(\mathcal O(h^{-3/2})\). Their pairing therefore yields
only an \(\mathcal O(h^{1/2})\) consistency estimate. Although the
F\"uhrer--Heuer discretization provides \(\mathcal O(h^2)\) best
approximations for both variables, its direct application to polygonal
approximations gives only an \(\mathcal O(h^{1/2})\) \(L^2\)-error bound
for the bending moment. Arnold and Walker observed a half-order
moment rate for straight-sided HHJ discretizations
\cite[Section~6.2.2]{ArnoldWalker2020}. Their boundary degrees of freedom
include normal--normal moments.

On every discrete boundary edge \(e\), we eliminate the constant
effective-shear mode by imposing
\[
  \int_e q_n(\bm M_h)\,\mathrm ds=0.
\]
For the F\"uhrer--Heuer element, \(q_n(\bm M_h)|_e\) is affine. The
constraint eliminates its constant part and leaves a zero-mean linear
mode. This mode is odd about the midpoint of \(e\) and therefore pairs
only with the odd part of the displacement trace. This odd trace is of
order \(\mathcal O(h^3)\), one order smaller than the full trace. The
resulting boundary consistency error is therefore bounded by
\(\mathcal O(h^{3/2})\), both on exact straight pieces and on polygonal
approximations of curved pieces. This estimate is relevant for the class
considered here, where \(\Gamma\) contains a genuinely curved portion of
\(\mathcal O(1)\) length. The correction fixes one scalar degree of freedom
per boundary edge and, once the polygonal mesh is given, requires no further
information about \(\Gamma\).

The effective-shear constraint can be enforced by local edge corrections.
The nontrivial part of the stability analysis, however, is correcting the
normal--normal trace after canonical interpolation. A moment field
satisfying \(M_{nn}=0\) on \(\Gamma\) generally has an interpolant whose
normal--normal trace does not vanish on \(\Gamma_h\). Removing this defect
within \(\ker(\DivDiv)\) requires a global treatment of the constant trace
mode on each boundary component. Using a characterization of the discrete
\(\DivDiv\) kernel and a uniform discrete boundary lifting, we construct
\(h\)-uniformly stable corrections for both boundary traces. Since these
corrections lie in \(\ker(\DivDiv)\), they leave the discrete equilibrium
equation unchanged and preserve the commuting property of the canonical
interpolant. The construction yields a uniform inf--sup condition for the
constrained mixed method and an admissible interpolant for the error
analysis.

Building on this uniform stability, we prove an \(L^2\)-error estimate of
order \(h^{3/2}\) for the bending moment. This rate holds for inner
approximations and, under a mild assumption on the extended load, for
polygonal boundaries that cross the physical boundary. Under an
\(H^4\)-regularity assumption for the auxiliary problem, we obtain an
\(L^2\)-displacement error of order \(h^2\) without any inclusion relation
between \(\Omega_h\) and \(\Omega\). Standard local cubic postprocessing yields an \(L^2\)-error estimate of order \(h^{3/2}\) for the broken Hessian.
 The method applies directly to
nonconvex and multiply connected domains whose boundaries contain both
exact straight segments and polygonal approximations of curved pieces,
without distinguishing between them in the algorithm.
The final numerical example combines all these geometric features.

The remainder is organized as follows. Section~\ref{sec:preliminaries}
presents the mixed formulation and the continuous and discrete geometric setting.
Section~\ref{sec:discrete-method} introduces the F\"uhrer--Heuer element
and the corrected mixed method. Section~\ref{sec:boundary-correction}
constructs the \(\DivDiv\)-preserving boundary corrections using the
uniform discrete boundary lifting established in
Appendix~\ref{app:boundary-lifting}. Section~\ref{sec:stability-error}
establishes discrete well-posedness and derives the error estimates.
Section~\ref{sec:supercloseness-postprocessing} analyzes the local cubic
postprocessing, and Section~\ref{sec:numerical-experiments} presents the
numerical experiments.
 \section{Mixed formulation and geometric setting}
\label{sec:preliminaries}

This section presents the mixed formulation and the geometric
setting for its polygonal discretization.

\subsection{Mixed formulation}
\label{subsec:mixed-formulation}

Let \(\omega\subset\mathbb R^2\) be a bounded Lipschitz domain.  For
\(s\geq0\), we denote by \(H^s(\omega)\) the usual Sobolev space and
use \((\cdot,\cdot)_\omega\) and \(\|\cdot\|_{s,\omega}\) for the
\(L^2(\omega)\) inner product and the \(H^s(\omega)\) norm,
respectively.  We abbreviate \(\|\cdot\|_{0,\omega}\) by
\(\|\cdot\|_\omega\), with the same notation used componentwise for
vector- and tensor-valued functions.  We write
\(\mathbb S:=\{\bm N\in\mathbb R^{2\times2}:\bm N^\top=\bm N\}\) and set
\(\bm N:\bm Q:=\operatorname{tr}(\bm N^\top\bm Q)\) and
\(|\bm N|:=(\bm N:\bm N)^{1/2}\).  For a
scalar function \(v\), the symbols \(\nabla v\) and \(D^2v\) denote its
gradient and Hessian.  Whenever unit tangent and normal fields
\(\bm t\) and \(\bm n\) are specified, we write
\(\partial_t:=\bm t\cdot\nabla\) and
\(\partial_n:=\bm n\cdot\nabla\); the notation
\(\partial_{t_e}\) is used analogously for \(\bm t_e\).  On a straight
edge, repeated subscripts denote repeated directional differentiation,
so that \(\partial_{tt} v=\bm t^\top D^2v\,\bm t\).  Finally,
\(\mathbb{P}_k(\omega;X)\) is the space of
\(X\)-valued polynomials of total degree at most \(k\); the range is
omitted when \(X=\mathbb R\).  We write \(a\lesssim b\) if
\(a\leq Cb\), where the constant \(C\) is independent of the mesh size.

For a tensor field \(\bm N\), the divergence is understood row-wise, and
all derivatives below are distributional.  We set
\(
\DivDiv\bm N:=\Div(\Div\bm N)
\)
and define
\[
H(\DivDiv,\omega;\mathbb S) :=
\{
\bm N\in L^2(\omega;\mathbb S):
\DivDiv\bm N\in L^2(\omega)
\},
\]
with graph norm $\|\bm N\|_{H(\DivDiv,\omega)}^2
:=
\|\bm N\|_\omega^2+
\|\DivDiv\bm N\|_\omega^2$.
For \(\bm p=(p_1,p_2)^\top\in H^1(\omega;\mathbb R^2)\), we denote
\[
\Curl\bm p
:=
\begin{pmatrix}
\partial_2 p_1 & -\partial_1 p_1\\
\partial_2 p_2 & -\partial_1 p_2
\end{pmatrix},
\qquad
\sCurl\bm p
:=
\frac12\bigl(\Curl\bm p+(\Curl\bm p)^\top\bigr).
\]
In particular,
\(
\DivDiv(\sCurl\bm p)=0
\)
in the sense of distributions.

The precise assumptions on the plate domain and its polygonal
approximations are specified in
Section~\ref{subsec:continuous-discrete-geometry}. For the bending
operator in \eqref{eq:intro-strong-problem}, define the compliance operator by
\begin{equation}
\label{eq:compliance-operator}
\mathbb A_\sigma\bm Q
:=\mathbb C_\sigma^{-1}\bm Q
=
\frac{1}{1-\sigma}
\Big(
\bm Q-\frac{\sigma}{1+\sigma}\operatorname{tr}(\bm Q)\bm I
\Big).
\end{equation}
It satisfies
\((1+\sigma)^{-1}|\bm Q|^2
\leq(\mathbb A_\sigma\bm Q):\bm Q
\leq(1-\sigma)^{-1}|\bm Q|^2\)
and is therefore uniformly bounded and coercive for
\(0\leq\sigma\leq\sigma_\ast<1\).

Set
\(V(\Omega):=H^2(\Omega)\cap H_0^1(\Omega)\).
Define the continuous moment space by
\begin{equation}
\label{eq:continuous-moment-space}
\Sigma(\Omega)
:=
\{
\bm N\in H(\DivDiv,\Omega;\mathbb S):
(\DivDiv\bm N,v)_\Omega-(\bm N,D^2v)_\Omega=0 ~\forall v\in V(\Omega)
\}.
\end{equation}
The constraint in \eqref{eq:continuous-moment-space} imposes
\(N_{nn}=0\) on \(\Gamma\) in the generalized trace sense and coincides
with the classical pointwise condition when \(\Gamma\) and \(\bm N\) are smooth.
Using this moment space, the mixed formulation of
\eqref{eq:intro-strong-problem} is to find
\((\bm M,u)\in\Sigma(\Omega)\times L^2(\Omega)\) such that
\begin{subequations}
\label{eq:continuous-mixed-problem}
\begin{alignat}{2}
(\mathbb A_\sigma\bm M,\bm N)_\Omega
-(u,\DivDiv\bm N)_\Omega
&=0
&\quad&\forall \bm N\in\Sigma(\Omega),\\
(\DivDiv\bm M,v)_\Omega
&=(f,v)_\Omega
&\quad&\forall v\in L^2(\Omega).
\end{alignat}
\end{subequations}

\subsection{Continuous and discrete geometry}
\label{subsec:continuous-discrete-geometry}
We assume that \(\Omega\subset\mathbb R^2\) is a bounded, connected
Lipschitz domain whose boundary has finitely many connected components,
\(\Gamma:=\partial\Omega=\bigcup_{i=0}^{H}\Gamma^{(i)}\). Here
\(\Gamma^{(0)}\) is the exterior boundary and the remaining components
bound the holes. 

For the geometric consistency estimates, each
\(\Gamma^{(i)}\) is further assumed to be a {\it simple closed piecewise
\(C^{1,1}\) curve}. Let \(\mathcal V_\Gamma\) be the finite set of
geometric corners and \(\mathcal C_\Gamma\) the collection of relatively
open \(C^{1,1}\) pieces of \(\Gamma\setminus\mathcal V_\Gamma\). A boundary
component without corners is included in \(\mathcal C_\Gamma\) as a whole.
We orient each \(\Gamma^{(i)}\) so that \(\Omega\) lies to its right and
denote the corresponding tangent and outward normal on each
\(\zeta\in\mathcal C_\Gamma\) by \(\bm t\) and \(\bm n\), with one-sided
values at geometric corners.

\begin{figure}[!htbp]
  \centering
  \includegraphics[width=0.64\textwidth]{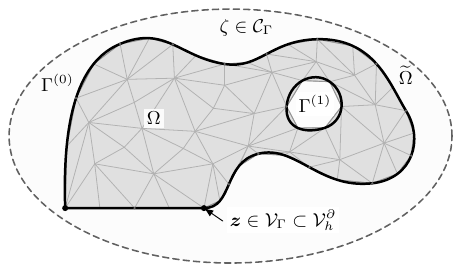}
  \caption{Continuous and discrete geometry. The gray triangulation represents
  \(\Omega_h\). Its boundary vertices lie on \(\Gamma\), geometric corners
  belong to \(\mathcal V_h^\partial\), and \(\widetilde\Omega\) is fixed
  independently of \(h\).}
  \label{fig:domain-geometry}
\end{figure}

\paragraph{Discrete geometry}
For every sufficiently small \(h\), let \(\Omega_h\) be a bounded
polygonal Lipschitz domain with boundary
\(\Gamma_h=\bigcup_{i=0}^{H}\Gamma_h^{(i)}\), and let \(\mathcal T_h\)
be a conforming triangulation of \(\Omega_h\). We
denote by \(\mathcal E_h\), \(\mathcal E_h^\partial\), \(\mathcal V_h\),
and \(\mathcal V_h^\partial\) the sets of edges, boundary edges, vertices,
and boundary vertices. Set \(h_K:=\operatorname{diam}K\), \(h_e:=|e|\),
and \(h:=\max_{K\in\mathcal T_h}h_K\). We orient
\(\Gamma_h^{(i)}\) as \(\Gamma^{(i)}\) and denote the constant unit
tangent and outward normal on \(e\in\mathcal E_h^\partial\) by
\(\bm t_e\) and \(\bm n_e\).

\begin{assumption}[polygonal boundary-interpolating meshes]
\label{ass:polygonal-boundary}
The family \(\{\mathcal T_h\}_h\) is uniformly shape regular and
quasi-uniform, with \(h\to0\) and \(\Gamma_h\to\Gamma\) in Hausdorff
distance. On each component, the vertices of \(\Gamma_h^{(i)}\) lie on
\(\Gamma^{(i)}\) in the same cyclic order, with consecutive vertices
joined by boundary edges. Thus, \(\mathcal V_h^\partial\subset\Gamma\).
No inclusion relation between \(\Omega_h\) and \(\Omega\) is required.
\end{assumption}

The Lipschitz regularity of \(\Omega\), together with
Assumption~\ref{ass:polygonal-boundary}, implies that
the domains \(\Omega_h\), \(0<h\le h_0\), form a uniformly Lipschitz
family. 
For the geometric consistency estimates, we further assume that the mesh resolves all corners, i.e., 
\begin{equation}\label{eq:corner-resolution}
\mathcal V_\Gamma\subset\mathcal V_h^\partial.
\end{equation}

\paragraph{Fixed background domain}
Since \(\Omega_h\) and \(\Omega\) need not be nested, choose a fixed bounded
smooth domain \(\widetilde\Omega\) such that
\(\overline\Omega\cup\overline{\Omega_h}\Subset\widetilde\Omega\)
for \(0<h\le h_0\).
For a fixed sufficiently small \(\delta_\Gamma>0\), set
\(\mathcal U:=\{\bm x:\operatorname{dist}(\bm x,\Gamma)<\delta_\Gamma\}
\Subset\widetilde\Omega\).
For \(h\) small enough, the region between the boundaries lies in \(\mathcal U\).
The fixed Lipschitz charts give the standard trace bound
(cf. \cite{Grisvard1985})
\begin{equation}
 \|v\|_{0,\Gamma}+\|v\|_{0,\Gamma_h}
 \lesssim\|v\|_{H^1(\widetilde\Omega)} \quad 
\forall v\in H^1(\widetilde\Omega),
 \label{eq:uniform-trace}
\end{equation}
with an \(h\)-independent constant.

For \(e=[\bm{z}_0,\bm{z}_1]\subset\Gamma_h^{(i)}\), let
\(\Gamma_e\subset\Gamma^{(i)}\) denote the physical boundary segment
from \(\bm{z}_0\) to \(\bm{z}_1\) in the common orientation. By
\eqref{eq:corner-resolution}, it lies in the closure of one \(C^{1,1}\)
boundary piece.
For \(h\) small enough, each pair \((e,\Gamma_e)\) admits the normal-graph
representation
\[
\Gamma_e=\Phi_e(e),
\quad
\Phi_e(\bm x):=\bm x+\rho_e(\bm x)\bm n_e,
\quad
\rho_e\in W^{1,\infty}(e),
\quad
\rho_e|_{\partial e}=0,
\]
where \(\|\rho_e\|_{L^\infty(e)}\lesssim h_e^2\) and
\(\|\partial_{t_e}\rho_e\|_{L^\infty(e)}\lesssim h_e\).
Writing \(\rho_e':=\partial_{t_e}\rho_e\), the corresponding frame is
\begin{equation}
 \bm t\circ\Phi_e
 =\frac{\bm t_e+\rho_e'\bm n_e}{\sqrt{1+|\rho_e'|^2}},
 \qquad
 \bm n\circ\Phi_e
 =\frac{\bm n_e-\rho_e'\bm t_e}{\sqrt{1+|\rho_e'|^2}}.
 \label{eq:boundary-frame-formulas}
\end{equation}
Consequently, with one-sided values used at geometric corners,
\(\|\bm n\circ\Phi_e-\bm n_e\|_{L^\infty(e)}
+\|\bm t\circ\Phi_e-\bm t_e\|_{L^\infty(e)}\lesssim h_e\).
For \(j=0,1\), the standard chain rule gives
\begin{equation}
 \|v\circ\Phi_e\|_{H^j(e)}
 \simeq\|v\|_{H^j(\Gamma_e)}
 \quad\forall v\in H^j(\Gamma_e),
 \label{eq:boundary-pullback-estimates}
\end{equation}
with constants independent of \(e\) and \(h\).

\paragraph{Boundary-strip}
For \(e\in\mathcal E_h^\partial\), define
\[
 \mathcal S_e:=\{\bm x+r\bm n_e:\bm x\in e,\ 
       \min(0,\rho_e(\bm x))<r<\max(0,\rho_e(\bm x))\}.
\]
For sufficiently small \(h\), these sets are pairwise disjoint and lie
in \(\mathcal U\).
All strip estimates below allow \(\rho_e\) to change sign, and  integration
over \(\mathcal S_e\) uses the usual area measure.

\begin{lemma}[boundary-strip estimates]
\label{lem:boundary-strip-estimates}
For every \(v\in H^1(\widetilde\Omega)\) and
\(e\in\mathcal E_h^\partial\),
\begin{align}
 \|v|_e-v\circ\Phi_e\|_{0,e}^2
 &\lesssim h_e^2\|\partial_{n_e}v\|_{0,\mathcal S_e}^2,
 \label{eq:boundary-trace-transfer}\\
 \|v\|_{0,\mathcal S_e}^2
 &\lesssim h_e^2\|v\|_{0,e}^2
       +h_e^4\|\partial_{n_e}v\|_{0,\mathcal S_e}^2,
 \label{eq:boundary-strip-estimate}
\end{align}
where \(v\circ\Phi_e\) denotes the pullback of the trace on
\(\Gamma_e\). The hidden constants are independent of \(e\) and \(h\).
\end{lemma}

\begin{proof}
Fix an oriented edge \(e=[\bm z_0,\bm z_1]\) and write
\(\bm x(s):=\bm z_0+s\bm t_e\), \(0<s<h_e\), for its
parametrization. Set \(\rho(s):=\rho_e(\bm x(s))\) and parametrize the
strip by \(\bm y(s,\theta):=\bm x(s)+\theta\rho(s)\bm n_e\),
\(0<\theta<1\).
Since \(\bm t_e\) and \(\bm n_e\) are orthonormal, the absolute
Jacobian is \(|\rho(s)|\).
For smooth \(v\), set \(w(s,\theta):=v(\bm y(s,\theta))\).
The fundamental theorem of calculus along each normal segment gives
\(w(s,\theta)-w(s,0)
=\rho(s)\int_0^\theta(\partial_{n_e}v)(\bm y(s,\tau))\,d\tau\).
Taking \(\theta=1\), applying Cauchy--Schwarz, and integrating in \(s\),
we obtain
\[
\begin{aligned}
 \|v|_e-v\circ\Phi_e\|_{0,e}^2
 &\le\int_0^{h_e}|\rho(s)|^2\int_0^1
       |(\partial_{n_e}v)(\bm y(s,\tau))|^2\,d\tau\,ds\\
 &\le\|\rho_e\|_{L^\infty(e)}
       \|\partial_{n_e}v\|_{0,\mathcal S_e}^2.
\end{aligned}
\]
Thus \(\|\rho_e\|_{L^\infty(e)}\lesssim h_e^2\) proves
\eqref{eq:boundary-trace-transfer}.
For \(0<\theta<1\), the same identity gives
\[
 |w(s,\theta)|^2
 \le 2|w(s,0)|^2+2|\rho(s)|^2\int_0^1
       |(\partial_{n_e}v)(\bm y(s,\tau))|^2\,d\tau.
\]
Multiplying by \(|\rho(s)|\) and integrating in \(s\) and \(\theta\)
yields
\[
 \|v\|_{0,\mathcal S_e}^2
 \le 2\|\rho_e\|_{L^\infty(e)}\|v\|_{0,e}^2
     +2\|\rho_e\|_{L^\infty(e)}^2
       \|\partial_{n_e}v\|_{0,\mathcal S_e}^2,
\]
which proves \eqref{eq:boundary-strip-estimate}. The general case follows by density.
\end{proof}
 \section{Boundary-corrected mixed finite element method} \label{sec:discrete-method}

We first recall the
\(H(\DivDiv;\mathbb S)\)-conforming triangular element given by
F\"uhrer and Heuer \cite{FuehrerHeuer2025}, and then introduce the
boundary-constrained space used in the discrete method.

\subsection{The F\"uhrer--Heuer triangular element}
\label{subsec:fh-element}

Let \(K\in\mathcal T_h\), and denote its sets of edges and vertices by
\(\mathcal E(K)\) and \(\mathcal V(K)\), respectively. We orient
\(\partial K\) so that \(K\) lies to its right. On each open edge, let
\(\bm t\) and \(\bm n\) denote the resulting constant unit tangent and outward
normal. On a boundary edge \(e\), they coincide with \(\bm t_e\) and \(\bm n_e\)
fixed in Section~\ref{subsec:continuous-discrete-geometry}. For a sufficiently
regular symmetric tensor field \(\bm N\), define
\begin{equation}
\label{eq:boundary-quantities}
N_{nn}:=\bm n^\top\bm N\bm n,
\quad
N_{nt}:=\bm n^\top\bm N\bm t,
\quad
q_n(\bm N):=\bm n\cdot\Div\bm N+\partial_t N_{nt}.
\end{equation}
For \(\bm z\in\mathcal V(K)\), let \(e_{\bm z}^-\) and \(e_{\bm z}^+\) denote
the edges ending and starting at \(\bm z\), respectively, and define
\(c_{\bm z}(\bm N):=
(N_{nt}|_{e_{\bm z}^-})(\bm z)-(N_{nt}|_{e_{\bm z}^+})(\bm z)\).
Two integrations by parts, followed by tangential integration by
parts on every edge, give
\begin{align}
&(\DivDiv\bm N,v)_K-(\bm N,D^2v)_K
\notag\\
&\qquad=
\sum_{e\in\mathcal E(K)}
\left(
\int_e q_n(\bm N)v\,\mathrm ds
-
\int_e N_{nn}\partial_n v\,\mathrm ds
\right)
-
\sum_{\bm z\in\mathcal V(K)}c_{\bm z}(\bm N)v(\bm z),
\label{eq:second-green-identity}
\end{align}
for \(\bm N\in C^2(\overline K;\mathbb S)\) and \(v\in H^2(K)\). This
identity displays the edge and vertex traces underlying the degrees
of freedom of the \(H(\DivDiv;\mathbb S)\)-conforming element.

Let \(\bm x:=(x_1,x_2)^\top\).  The Raviart--Thomas space and the
local tensor space are
\begin{equation}
\label{eq:fh-local-space}
\operatorname{RT}^k(K)
:=
\mathbb P_k(K;\mathbb R^2)+\bm x\mathbb P_k(K),
\quad 
X(K)
:=
\operatorname{sym}
\bigl(
\operatorname{RT}^0(K)\otimes\operatorname{RT}^1(K)
\bigr).
\end{equation}
Here, \(U\otimes V\) denotes the span of
\(\{\bm\phi\bm\psi^\top:\bm\phi\in U,\ \bm\psi\in V\}\), and
\(\operatorname{sym}\bm Q:=(\bm Q+\bm Q^\top)/2\).

For \(e\in\mathcal E(K)\), let \(s\in[0,h_e]\) be the arclength measured
from its initial endpoint, and set
\(\ell_{0,e}:=1\) and \(\ell_{1,e}:=\ell_e:=2s/h_e-1\).
The degrees of freedom are
\begin{subequations}
\label{eq:fh-dofs}
\begin{align}
d_{e,k}^{nn}(\bm N)
&:=
\frac{(N_{nn},\ell_{k,e})_e}
     {(\ell_{k,e},\ell_{k,e})_e},
&&
e\in\mathcal E(K),\quad k=0,1,
\label{eq:fh-dofs-nn}
\\
d_{e,k}^{q}(\bm N)
&:=
(q_n(\bm N),\ell_{k,e})_e,
&&
e\in\mathcal E(K),\quad k=0,1,
\label{eq:fh-dofs-q}
\\
d_{\bm z}^c(\bm N)
&:=
c_{\bm z}(\bm N),
&&
\bm z\in\mathcal V(K).
\label{eq:fh-dofs-corner}
\end{align}
\end{subequations}

F\"uhrer and Heuer \cite[Proposition~4 and Theorem~5]{FuehrerHeuer2025}
proved that the degrees of freedom in \eqref{eq:fh-dofs} are unisolvent
for \(X(K)\). The local space satisfies
\(\mathbb P_1(K;\mathbb S)\subset X(K)\subset\mathbb P_3(K;\mathbb S)\)
and \(\DivDiv X(K)=\mathbb P_1(K)\). Moreover, for
\(\bm N\in X(K)\) and \(e\in\mathcal E(K)\),
\(N_{nn}|_e,q_n(\bm N)|_e\in\mathbb P_1(e)\).

Before imposing the boundary constraints, the global moment and
displacement spaces are
\[
\begin{aligned}
X(\mathcal T_h)
&:=\{\bm N_h\in H(\DivDiv,\Omega_h;\mathbb S):
\bm N_h|_K\in X(K)\ \forall K\in\mathcal T_h\},\\
V_h
&:=\{v_h\in L^2(\Omega_h):
v_h|_K\in\mathbb P_1(K)\ \forall K\in\mathcal T_h\}.
\end{aligned}
\]
Conformity in \(X(\mathcal T_h)\) is enforced by matching the edge traces
and balancing the interior corner forces.
Let \(P_h^1:L^2(\Omega_h)\to V_h\) denote the elementwise \(L^2\)-orthogonal
projection, and write \(H^s(\mathcal T_h;\mathbb S)\) for the broken
tensor-valued Sobolev space.

For \(1/2<r\leq1\), F\"uhrer and Heuer
\cite[Proposition~10]{FuehrerHeuer2025} define the canonical interpolant
\(\Pi_h^{\mathrm{dDiv}}\bm N\in X(\mathcal T_h)\), for
\(\bm N\in H(\DivDiv,\Omega_h;\mathbb S)\cap
H^{1+r}(\mathcal T_h;\mathbb S)\), by matching the degrees of freedom in
\eqref{eq:fh-dofs}. It is a projection onto \(X(\mathcal T_h)\) and satisfies
\begin{subequations}
\begin{align}
\DivDiv\Pi_h^{\mathrm{dDiv}}\bm N
&=
P_h^1\DivDiv\bm N,
\label{eq:canonical-commuting}
\\
\|\bm N-\Pi_h^{\mathrm{dDiv}}\bm N\|_{\Omega_h}
&\lesssim
h^{1+r}\|\bm N\|_{1+r,\Omega_h}.
\label{eq:canonical-approximation}
\end{align}
\end{subequations}
The approximation estimate requires
\(\bm N\in H^{1+r}(\Omega_h;\mathbb S)
\cap H(\DivDiv,\Omega_h;\mathbb S)\).

\subsection{Boundary constraints and discrete formulation}
\label{subsec:boundary-spaces-method}

The standard simply supported moment space
and its corrected subspace are
\begin{equation}
\label{eq:boundary-constrained-spaces}
\begin{aligned}
\Sigma_h^{nn}
&:=
\left\{
\bm N_h\in X(\mathcal T_h):
N_{h,nn}|_e=0
\quad\forall e\in\mathcal E_h^\partial
\right\},\\
\Sigma_h^{nn,q}
&:=
\left\{
\bm N_h\in\Sigma_h^{nn}:
\int_e q_n(\bm N_h)\,\mathrm ds=0
\quad\forall e\in\mathcal E_h^\partial
\right\}.
\end{aligned}
\end{equation}
Since \(q_n(\bm N_h)|_e\in\mathbb P_1(e)\), the additional constraint
removes only its constant mode. In terms of the degrees of freedom,
\(d_{e,0}^{nn}\) and \(d_{e,1}^{nn}\) are fixed on every boundary edge,
as is \(d_{e,0}^q\). The odd shear degree of freedom and all boundary
corner-force degrees of freedom remain free. Thus, every discrete boundary
edge receives the same treatment. The method is implemented directly
through the boundary degrees of freedom and requires neither a boundary
parametrization, curvature evaluation, nor a classification of the
physical boundary.

Let \(f\in L^2(\Omega)\) be the prescribed load, and choose a fixed
extension \(\widetilde f\in L^2(\widetilde\Omega)\) satisfying
\(\widetilde f|_\Omega=f\). We set
\begin{equation}
\label{eq:discrete-load}
f_h:=P_h^1(\widetilde f|_{\Omega_h})\in V_h.
\end{equation}
For inner approximations, this reduces to
\(f_h=P_h^1(f|_{\Omega_h})\). If \(\Omega_h\not\subset\Omega\), the
assembly additionally uses the chosen extension on
\(\Omega_h\setminus\Omega\). The zero extension is an admissible choice.
The discrete method is to find
\(
(\bm M_h,u_h)\in\Sigma_h^{nn,q}\times V_h
\)
such that
\begin{subequations}
\label{eq:discrete-method}
\begin{alignat}{2}
(\mathbb A_\sigma\bm M_h,\bm N_h)_{\Omega_h}
-
(u_h,\DivDiv\bm N_h)_{\Omega_h}
&=0
&\quad&
\forall \bm N_h\in\Sigma_h^{nn,q},
\label{eq:discrete-method-a}
\\
(\DivDiv\bm M_h,v_h)_{\Omega_h}
&=
(f_h,v_h)_{\Omega_h}
&\quad&
\forall v_h\in V_h.
\label{eq:discrete-method-b}
\end{alignat}
\end{subequations}

\begin{remark}[nonhomogeneous data]
\label{rem:nonhomogeneous-data}
For sufficiently regular data \(u=g_D\) and \(M_{nn}=g_{nn}\), choose an
extension \(\widetilde g_{nn}\) to a neighborhood of \(\Gamma\) and
prescribe the two normal--normal degrees of freedom on each boundary
edge by the \(\mathbb P_1(e)\)-moments of
\(\widetilde g_{nn}|_e\).
The extension need not be the closest-point extension. Since the
boundary vertices lie on \(\Gamma\), the displacement datum is
represented by its continuous piecewise affine interpolant \(g_{D,h}\)
on \(\Gamma_h\), defined by the vertex values of \(g_D\). It enters the
first equation through the effective-shear and corner terms in
\eqref{eq:second-green-identity}. The effective-shear constraint
\(\int_e q_n(\bm M_h)\,\mathrm ds=0\) remains unchanged on every $e \in \mathcal E_h^\partial$, since the associated displacement variations are homogeneous.
\end{remark}
 \section{Boundary corrections}
\label{sec:boundary-correction}

The discrete moment space carries the boundary restrictions
\[
  N_{h,nn}|_e=0
  \quad\forall e\in\mathcal E_h^\partial,
  \qquad
  \int_e q_n(\bm N_h)\,\mathrm ds=0
  \quad\forall e\in\mathcal E_h^\partial.
\]
Consequently, the stability and error analysis requires more than the
standard properties of the unconstrained F\"uhrer--Heuer space
\(X(\mathcal T_h)\). Starting from an arbitrary tensor in this space,
we construct two successive boundary corrections,
\begin{equation}
\label{eq:correction-chain}
  X(\mathcal T_h)
  \xrightarrow{\;\Pi_h^{nn}\;}
  \Sigma_h^{nn}
  \xrightarrow{\;\Pi_h^{nn,q}\;}
  \Sigma_h^{nn,q}.
\end{equation}
The first correction removes the two normal--normal moments on each
boundary edge. The second removes the zeroth-order effective-shear
moment on every boundary edge while preserving the normal--normal
condition. Both corrections are linear and preserve the discrete double
divergence. Their norms are controlled, with \(h\)-independent constants,
by the boundary moments being removed. These properties are
used twice below: first to transfer a stable right inverse of \(\DivDiv\)
to the corrected space, and then to turn the canonical interpolant into
an admissible comparison function for the error analysis. We follow the
order in \eqref{eq:correction-chain}: the two normal--normal modes are
removed first, and the effective-shear moment is removed afterwards.
The two corrections act on independent boundary modes and therefore
commute.

For \(e=[\bm{z}_i,\bm{z}_j]\in\mathcal E_h^\partial\), let \(K_e\) be
its adjacent element and \(\lambda_i,\lambda_j\) the corresponding
barycentric coordinates. Define the edge bubble \(b_e:=\lambda_i\lambda_j\)
on \(K_e\), extended by zero elsewhere. We use the edge orientations and
trace notation given in Section~\ref{sec:discrete-method}.

\begin{lemma}\label{lem:scurl-trace}
Let \(e\) be an oriented straight edge with unit tangent \(\bm t\) and unit
normal \(\bm n\). For every sufficiently smooth vector field \(\bm p\) in a
neighborhood of \(e\), one has
\begin{equation}\label{eq:scurl-trace}
  (\sCurl\bm p)_{nn}
  =-\partial_t(\bm p\cdot\bm n),
  \qquad
  q_n(\sCurl\bm p)
  =-\partial_{tt}(\bm p\cdot\bm t).
\end{equation}
\end{lemma}

\begin{proof}
Choose local coordinates such that
\(\bm t=(1,0)^\top\) and \(\bm n=(0,1)^\top\). Direct expansion of the symmetric Curl
gives $(\sCurl\bm p)_{nn}
  =-\partial_1p_2
  =-\partial_t(\bm p\cdot\bm n)$.
Substituting the components of \(\sCurl\bm p\) into
$\bm n\cdot\Div(\sCurl\bm p)+\partial_t(\sCurl\bm p)_{nt}$,
we find that the mixed derivatives cancel, leaving $q_n(\sCurl\bm p)
  =-\partial_{11}p_1
  =-\partial_{tt}(\bm p\cdot\bm t)$.
\end{proof}

For later use, let
\(\bm p_h\in C^0(\overline{\Omega}_h;\mathbb R^2)\) be piecewise quadratic,
that is, \(\bm p_h|_K\in \mathbb{P}_2(K;\mathbb R^2)\) for every
\(K\in\mathcal T_h\). Then the elementary inclusion reads
\begin{equation}
\label{eq:scurl-kernel}
  \sCurl\bm p_h\in X(\mathcal T_h)\cap\ker(\DivDiv).
\end{equation}

For \(\bm N_h\in X(\mathcal T_h)\), write its boundary data uniquely as
\begin{equation}\label{eq:holdall-three-data}
  \left.N_{h,nn}\right|_e=\mu_e^0+\mu_e^1\ell_e
  \quad \text{for }e\in\mathcal E_h^\partial,
  \qquad
  \mu_e^q:=\int_e q_n(\bm N_h)\,\mathrm{d}s
  \quad\text{for }e\in\mathcal E_h^\partial,
\end{equation}
where \(\ell_e\in \mathbb{P}_1(e)\) is the oriented odd polynomial introduced in
Section~\ref{subsec:fh-element}.  With the normalization used there,
it satisfies \(-\partial_t(h_eb_e)=\ell_e\).  This
decomposition identifies the three boundary defects to be corrected.
They are removed in two stages.

\subsection{Normal--normal boundary correction}
\label{sec:normal-normal-correction}

We first remove the constant and odd normal--normal modes \(\mu_e^0\)
and \(\mu_e^1\) in \eqref{eq:holdall-three-data}, while leaving \(\mu_e^q\) unchanged.

\subsubsection{Local correction of the odd mode}
\label{sec:odd-normal-correction}

Define
\(\bm p_h^1:=\sum_{e\in\mathcal E_h^\partial}h_e\mu_e^1b_e\bm n_e\).

\begin{lemma}[normal-normal odd-mode correction]
\label{lem:odd-normal-lifting}
The tensor correction generated by \(\bm p_h^1\) satisfies $ \sCurl\bm p_h^1
\in X(\mathcal T_h)\cap\ker(\DivDiv)$ and 
\begin{align}
  \left.(\sCurl\bm p_h^1)_{nn}\right|_e
  &=\mu_e^1\ell_e,
  \quad
  \left.q_n(\sCurl\bm p_h^1)\right|_e=0
  &&\forall e\in\mathcal E_h^\partial,
  \label{eq:odd-normal-traces}
  \\
  \lVert\sCurl\bm p_h^1\rVert_{L^2(\Omega_h)}
  &\lesssim
  \Big(
    \sum_{e\in\mathcal E_h^\partial}
    h_e^2|\mu_e^1|^2
  \Big)^{1/2}.
  \label{eq:odd-normal-bound}
\end{align}
\end{lemma}

\begin{proof}
Since \(\bm p_h^1\) is continuous and piecewise quadratic, its
symmetric curl belongs to \(X(\mathcal T_h)\). For a fixed boundary
edge \(e\), the \(e\)-th summand satisfies
\((h_e\mu_e^1b_e\bm n_e)\cdot\bm n_e=h_e\mu_e^1b_e\) and
\((h_e\mu_e^1b_e\bm n_e)\cdot\bm t_e=0\). Lemma~\ref{lem:scurl-trace} and
the identity \(-\partial_t(h_eb_e)=\ell_e\) therefore give the two identities in
\eqref{eq:odd-normal-traces}; all other summands vanish on \(e\).

Finally, shape regularity, finite overlap, and
\(\lVert\nabla b_e\rVert_{L^\infty(K_e)}\lesssim h_e^{-1}\) give
\[
  \lVert\sCurl\bm p_h^1\rVert_{L^2(\Omega_h)}^2
  \lesssim
  \sum_{e\in\mathcal E_h^\partial}
  \left\lVert
    \sCurl(h_e\mu_e^1b_e\bm n_e)
  \right\rVert_{L^2(K_e)}^2
  \lesssim
  \sum_{e\in\mathcal E_h^\partial}h_e^2|\mu_e^1|^2.
\]
This proves \eqref{eq:odd-normal-bound}.
\end{proof}

\subsubsection{Global correction of the constant mode}
\label{sec:constant-normal-correction}

Let \(S_h^1:=\{p_h\in C^0(\overline\Omega_h):
p_h|_K\in\mathbb P_1(K)\ \forall K\in\mathcal T_h\}\).
The preceding edge bubbles cannot correct the constant normal--normal mode.
We enforce closure around each component using stable tangential increments,
then lift the boundary potential into \([S_h^1]^2\), keeping the correction
in \(\ker(\DivDiv)\).

\begin{lemma}[normal--normal constant-mode correction]
\label{lem:distributed-normal-lifting}
Let \(\mu^0=(\mu_e^0)_{e\in\mathcal E_h^\partial}\) be a family of constants
assigned to the boundary edges. Then there exists \(\bm p_h^0\in [S_h^1]^2\)
such that $  \sCurl\bm p_h^0 \in X(\mathcal T_h)\cap\ker(\DivDiv)$ and 
\begin{align}
  \left.(\sCurl\bm p_h^0)_{nn}\right|_e
  =\mu_e^0,
  \qquad
  \left.q_n(\sCurl\bm p_h^0)\right|_e=0
  &&\text{for every }e\in\mathcal E_h^\partial,
  \label{eq:distributed-normal-traces}
  \\
  \lVert\sCurl\bm p_h^0\rVert_{L^2(\Omega_h)}
  \lesssim
  \Big(
    \sum_{e\in\mathcal E_h^\partial}
    h_e|\mu_e^0|^2
  \Big)^{1/2}.
  \label{eq:distributed-normal-bound}
\end{align}
Moreover, the map \(\mu^0\mapsto\bm p_h^0\) is linear.
\end{lemma}
\begin{proof}
The construction is carried out independently on each connected component
of \(\Gamma_h\) and then assembled over the finitely many components. In
what follows, we work on one fixed connected component of \(\Gamma_h\).

Use the orientation fixed in Section~\ref{subsec:continuous-discrete-geometry},
label its vertices by
\(\bm{z}_0,\bm{z}_1,\ldots,\bm{z}_J=\bm{z}_0\), and write
\(e_i=[\bm{z}_i,\bm{z}_{i+1}]\), \(i=0,\ldots,J-1\). Let
\(h_i:=|e_i|\), \(\mu_i^0:=\mu_{e_i}^0\), and denote by \(\bm t_i\) and
\(\bm n_i\) the unit tangent and outward unit normal to \(e_i\). Set
\(A_h^2:=\sum_{i=0}^{J-1}h_i|\mu_i^0|^2\).

\medskip\noindent
{\it Step 1: boundary nodal increments.}
We first seek a continuous piecewise affine boundary potential
\(\bm p_h^\partial\). Let \(\bm p_i:=\bm p_h^\partial(\bm{z}_i)\in\mathbb R^2\) denote its
unknown vertex values, understood up to a common additive vector. Since
\(\bm p_h^\partial|_{e_i}\in \mathbb{P}_1(e_i;\mathbb R^2)\), one has
\(\partial_t\bm p_h^\partial|_{e_i}=(\bm p_{i+1}-\bm p_i)/h_i\) and
\(\partial_{tt}(\bm p_h^\partial\cdot\bm t_i)|_{e_i}=0\). Consequently,
\eqref{eq:scurl-trace} shows that the prescribed constant normal--normal
trace is equivalent to
\((\bm p_{i+1}-\bm p_i)\cdot\bm n_i=-h_i\mu_i^0\). Writing the free tangential
component as \(h_i\alpha_i\), the general admissible increment is
\begin{equation}
\label{eq:distributed-increments}
  \bm d_i
  :=
  \bm p_{i+1}-\bm p_i
  =
  -h_i\mu_i^0\bm n_i+h_i\alpha_i\bm t_i,
\end{equation}
where \(\alpha_i\) remains to be determined. For the vertex values to be
single-valued on the closed component, the increments must satisfy
\(\sum_{i=0}^{J-1}\bm d_i=0\). By
\eqref{eq:distributed-increments}, it is equivalent to
\begin{equation}
\label{eq:distributed-closure}
  \sum_{i=0}^{J-1}h_i\alpha_i\bm t_i
  =
  \sum_{i=0}^{J-1}h_i\mu_i^0\bm n_i
  =:\bm r_h.
\end{equation}
Since the component length is uniformly bounded, Cauchy--Schwarz gives
\(|\bm r_h|\le  (\sum_{i=0}^{J-1}h_i|\mu_i^0|^2)^{1/2}(\sum_{i=0}^{J-1}h_i)^{1/2}\lesssim A_h\).

Fix three cyclically ordered noncollinear points on the corresponding
physical boundary component. After a cyclic relabeling, a standard
perturbation argument gives, for all sufficiently small \(h\), indices
\(0\le j_1<j_2<j_3\le J-1\) such that
\(\bm g_1:=\bm{z}_{j_2}-\bm{z}_{j_1}\) and \(\bm g_2:=\bm{z}_{j_3}-\bm{z}_{j_2}\) form a
uniformly nonsingular basis of \(\mathbb R^2\). Since \(\Gamma\) has
only finitely many boundary components, there are constants
\(c_0,C_0,\delta_0>0\), independent of \(h\) and of the component, such that
\[
  c_0\le |\bm g_k|\le C_0,
  \quad k=1,2,
  \qquad
  |\det(\bm g_1,\bm g_2)|\ge\delta_0|\bm g_1||\bm g_2|.
\]
Uniform nonsingularity yields unique coefficients \((\beta_1,\beta_2)\)
satisfying
\begin{equation}\label{eq:distributed-beta}
  \beta_1\bm g_1+\beta_2\bm g_2
  =\bm r_h,
  \qquad
  |\beta_1|+|\beta_2|\lesssim |\bm r_h|\lesssim A_h.
\end{equation}

Set
\(\mathcal I_1:=\{j_1,\ldots,j_2-1\}\) and
\(\mathcal I_2:=\{j_2,\ldots,j_3-1\}\). 
Since
\(h_i\bm t_i=\bm{z}_{i+1}-\bm{z}_i\), telescoping gives
\(\bm g_k=\sum_{i\in\mathcal I_k}h_i\bm t_i\), \(k=1,2\). Set
\begin{equation}
\label{eq:distributed-alpha}
  \alpha_i
  :=
  \begin{cases}
    \beta_k,&i\in\mathcal I_k,\quad k=1,2,\\
    0,&\text{otherwise}.
  \end{cases}
\end{equation}
Equations \eqref{eq:distributed-beta} and
\eqref{eq:distributed-alpha} then imply \eqref{eq:distributed-closure}. Using
again the uniform bound on the component length gives
\begin{equation}
\label{eq:distributed-alpha-bound}
  \sum_{i=0}^{J-1}h_i|\alpha_i|^2
  =
  |\beta_1|^2\sum_{i\in\mathcal I_1}h_i
  +|\beta_2|^2\sum_{i\in\mathcal I_2}h_i
  \lesssim
  A_h^2.
\end{equation}

\medskip\noindent
{\it Step 2: construction of the boundary potential.}
Thanks to \(\sum_{i=0}^{J-1}\bm d_i=0\), the
nodal increments \(\bm p_h^\partial(\bm{z}_{i+1})-\bm p_h^\partial(\bm{z}_i)=\bm d_i\) define a
continuous piecewise affine vector field \(\bm p_h^\partial\). We fix the
remaining additive vector by requiring \(\bm p_h^\partial\) to have zero mean.
Since \(\bm p_h^\partial\) is affine on \(e_i\),
\(\partial_t\bm p_h^\partial|_{e_i}=\bm d_i/h_i=-\mu_i^0\bm n_i+\alpha_i\bm t_i\).
Thus, by \eqref{eq:distributed-alpha-bound},
\[
  \sum_{i=0}^{J-1}
  \|\partial_t\bm p_h^\partial\|_{L^2(e_i)}^2
  =\sum_{i=0}^{J-1}h_i\bigl(|\mu_i^0|^2+|\alpha_i|^2\bigr)
  \lesssim
  A_h^2.
\]
Summing the componentwise estimate gives $\|\partial_t\bm p_h^\partial\|_{L^2(\Gamma_h)}
  \lesssim
  (
    \sum_{e\in\mathcal E_h^\partial}
    h_e|\mu_e^0|^2
  )^{1/2}$.

\medskip\noindent
{\it Step 3: discrete lifting and conclusion.}
Applying Lemma~\ref{lem:discrete-trace-lifting} to each scalar component yields
\(\bm p_h^0\in [S_h^1]^2\) with
\(\bm p_h^0|_{\Gamma_h}=\bm p_h^\partial\) and
\begin{equation}
\label{eq:distributed-potential-energy}
  \|\bm p_h^0\|_{H^1(\Omega_h)}
  \lesssim
  \|\partial_t\bm p_h^\partial\|_{L^2(\Gamma_h)}
  \lesssim
  \Big(
    \sum_{e\in\mathcal E_h^\partial}
    h_e|\mu_e^0|^2
  \Big)^{1/2}.
\end{equation}

The kernel inclusion follows from \eqref{eq:scurl-kernel}.
Lemma~\ref{lem:scurl-trace} and the construction
\eqref{eq:distributed-increments} give \eqref{eq:distributed-normal-traces}.
Finally, \eqref{eq:distributed-potential-energy} implies
\eqref{eq:distributed-normal-bound}.
\end{proof}

For \(\bm N_h\in X(\mathcal T_h)\) with boundary coefficients
given by \eqref{eq:holdall-three-data}, using Lemmas~\ref{lem:odd-normal-lifting} and~\ref{lem:distributed-normal-lifting}, we now define
\(\Pi_h^{nn}:X(\mathcal T_h)\to\Sigma_h^{nn}\) by
\begin{equation}\label{eq:two-level-corrections}
  \Pi_h^{nn}\bm N_h
  :=
  \bm N_h-\sCurl\bigl(\bm p_h^1+\bm p_h^0\bigr),
\end{equation}
which is a linear projection satisfying
\(\DivDiv(\Pi_h^{nn}\bm N_h)=\DivDiv\bm N_h\).
A standard scaling argument gives
\(\sum_{e\in\mathcal E_h^\partial}h_e^2|\mu_e^1|^2
\lesssim\sum_{e\in\mathcal E_h^\partial}h_e
\|N_{h,nn}\|_{L^2(e)}^2
\lesssim\|\bm N_h\|_{L^2(\Omega_h)}^2\). Together with
\eqref{eq:odd-normal-bound} and \eqref{eq:distributed-normal-bound}, this yields
\begin{equation}\label{eq:normal-correction-bound}
  \lVert\Pi_h^{nn}\bm N_h\rVert_{H(\DivDiv,\Omega_h)}
  \lesssim
  \lVert\bm N_h\rVert_{H(\DivDiv,\Omega_h)}
  +
  \Big(
    \sum_{e\in\mathcal E_h^\partial}
    h_e|\mu_e^0|^2
  \Big)^{1/2}.
\end{equation}

\subsection{Effective-shear boundary correction}
\label{sec:effective-shear-correction}

As in Section~\ref{sec:odd-normal-correction}, for
\(\bm N_h\in X(\mathcal T_h)\) with coefficients from
\eqref{eq:holdall-three-data}, define
\(\bm p_h^q:=\frac12\sum_{e\in\mathcal E_h^\partial}
h_e\mu_e^q b_e\bm t_e\).

\begin{lemma}[effective-shear moment correction]
\label{lem:shear-correction}
The tensor correction satisfies
\(\sCurl\bm p_h^q\in\Sigma_h^{nn}\cap\ker(\DivDiv)\) and
\begin{align}
  \left.(\sCurl\bm p_h^q)_{nn}\right|_e
  &=0,
  \qquad
  \int_e q_n(\sCurl\bm p_h^q)\,\mathrm ds=\mu_e^q
  &&\forall e\in\mathcal E_h^\partial,
  \label{eq:shear-correction-traces}\\
  \|\sCurl\bm p_h^q\|_{L^2(\Omega_h)}
  &\lesssim
  \Big(\sum_{e\in\mathcal E_h^\partial}h_e^2|\mu_e^q|^2\Big)^{1/2}.
  \label{eq:shear-correction-bound}
\end{align}
\end{lemma}

\begin{proof}
The inclusion and \eqref{eq:shear-correction-traces} follow from
\eqref{eq:scurl-kernel} and Lemma~\ref{lem:scurl-trace}, using
\(-\partial_{tt}b_e=2h_e^{-2}\) on \(e\). Standard scaling and finite
overlap yield \eqref{eq:shear-correction-bound}.
\end{proof}

Define \(\Pi_h^{nn,q}\bm N_h:=\bm N_h-\sCurl\bm p_h^q\).
By Lemma~\ref{lem:shear-correction}, this operator preserves the
normal--normal trace and \(\DivDiv\), and its restriction
\(\Pi_h^{nn,q}:\Sigma_h^{nn}\to\Sigma_h^{nn,q}\) is a projection.
A standard scaling argument gives
\(\sum_{e\in\mathcal E_h^\partial}h_e^2|\mu_e^q|^2
\le\sum_{e\in\mathcal E_h^\partial}h_e^3
\|q_n(\bm N_h)\|_{L^2(e)}^2\). The latter is bounded by
\(\|\bm N_h\|_{L^2(\Omega_h)}^2\). Together with
\eqref{eq:shear-correction-bound} and preservation of \(\DivDiv\), this yields
\begin{equation}\label{eq:projection-stability}
  \|\Pi_h^{nn,q}\bm N_h\|_{H(\DivDiv,\Omega_h)}
  \lesssim\|\bm N_h\|_{H(\DivDiv,\Omega_h)}.
\end{equation}

The two projections commute on \(X(\mathcal T_h)\), since each preserves
the boundary data corrected by the other. Their composition satisfies
\begin{equation}\label{eq:holdall-full-correction}
  \Pi_h^{nn,q}\Pi_h^{nn}\bm N_h\in\Sigma_h^{nn,q},
  \qquad
  \DivDiv(\Pi_h^{nn,q}\Pi_h^{nn}\bm N_h)=\DivDiv\bm N_h,
\end{equation}
with its \(H(\DivDiv,\Omega_h)\) bound following from
\eqref{eq:normal-correction-bound} and \eqref{eq:projection-stability}.
Finally, \eqref{eq:odd-normal-bound},
\eqref{eq:distributed-normal-bound}, and \eqref{eq:shear-correction-bound} give
\begin{equation}\label{eq:holdall-correction-bound}
  \|\bm N_h-\Pi_h^{nn,q}\Pi_h^{nn}\bm N_h\|_{L^2(\Omega_h)}
  \lesssim
  \Bigl[\sum_{e\in\mathcal E_h^\partial}
  \bigl(h_e|\mu_e^0|^2
       +h_e^2|\mu_e^1|^2
       +h_e^2|\mu_e^q|^2\bigr)\Bigr]^{1/2}.
\end{equation}
\section{Discrete well-posedness and error analysis}
\label{sec:stability-error}

This section first establishes the discrete well-posedness of the mixed
method. We then estimate the boundary consistency error and the
approximation error of the constrained interpolant, leading to the
bending-moment error estimate. Finally, under an \(H^4\)-regularity
assumption for the dual problem, a duality argument yields the optimal
\(L^2\)-error estimate for the displacement.
We use the fixed background domain \(\widetilde\Omega\) throughout.

\subsection{Discrete well-posedness} \label{subsec:abstract-error}

\begin{theorem}[discrete inf-sup]
\label{thm:infsup}
Suppose that Assumption~\ref{ass:polygonal-boundary} holds. Then there exists a constant
\(C>0\), independent of \(h\), such that for every
\(v_h\in V_h\), there exists \(\bm N_h^{nn,q}\in\Sigma_h^{nn,q}\) satisfying
\begin{equation}\label{eq:stable-infsup-witness}
  \begin{aligned}
    \DivDiv\bm N_h^{nn,q}&=v_h,\\
    \|\bm N_h^{nn,q}\|_{H(\DivDiv,\Omega_h)}
    &\le
    C\|v_h\|_{L^2(\Omega_h)}.
  \end{aligned}
\end{equation}
Consequently, the pair $\Sigma_h^{nn,q}\times V_h$ satisfies a
uniform discrete inf--sup condition.
\end{theorem}

\begin{proof}
For \(v_h\in V_h\), let \(\widetilde v_h\) denote its zero extension to
\(\widetilde\Omega\), and let \(z\in H_0^1(\widetilde\Omega)\) solve
\[
  -\Delta z=\widetilde v_h
  \quad\text{in }\widetilde\Omega,
  \qquad
  z=0
  \quad\text{on }\partial\widetilde\Omega.
\]
The fixed smooth domain \(H^2\)-regularity estimate for the above problem gives
\begin{equation}\label{eq:holdall-poisson-bound}
  \|z\|_{H^2(\widetilde\Omega)}
  \lesssim
  \|\widetilde v_h\|_{L^2(\widetilde\Omega)}
  =
  \|v_h\|_{L^2(\Omega_h)}.
\end{equation}
Restricting \(-z\bm I\) to \(\Omega_h\), define
\begin{equation}\label{eq:infsup-Nh-construction}
  \bm N_h:=\Pi_h^{\mathrm{dDiv}}\bigl((-z\bm I)|_{\Omega_h}\bigr),
  \qquad
  \bm N_h^{nn,q}:=\Pi_h^{nn,q}\Pi_h^{nn}\bm N_h.
\end{equation}
By \eqref{eq:canonical-commuting} and \eqref{eq:holdall-full-correction},
the corrected tensor in \eqref{eq:infsup-Nh-construction} belongs to
\(\Sigma_h^{nn,q}\) and satisfies
\(\DivDiv\bm N_h^{nn,q}=\DivDiv\bm N_h=v_h\).
Moreover, \eqref{eq:canonical-approximation},
\eqref{eq:canonical-commuting}, and \eqref{eq:holdall-poisson-bound}
yield the estimate $\|\bm N_h\|_{H(\DivDiv,\Omega_h)}
  \lesssim\|z\|_{H^2(\widetilde\Omega)}
  \lesssim\|v_h\|_{L^2(\Omega_h)}$.

For the constant coefficients \(\mu_e^0\) of \(\bm N_h\) in
\eqref{eq:holdall-three-data}, preservation of the normal--normal
moments gives \(\mu_e^0=-h_e^{-1}\int_e z\,\mathrm ds\).
Thus Cauchy--Schwarz, \(|\Gamma_h|\lesssim1\), and the embedding
\(H^2(\widetilde\Omega)\hookrightarrow L^\infty(\widetilde\Omega)\) imply
\begin{equation}\label{eq:holdall-uniform-trace}
  \sum_{e\in\mathcal E_h^\partial}h_e|\mu_e^0|^2
  \leq\|z\|_{L^2(\Gamma_h)}^2
  \leq|\Gamma_h|\|z\|_{L^\infty(\widetilde\Omega)}^2
  \lesssim\|z\|_{H^2(\widetilde\Omega)}^2.
\end{equation}
Combining \eqref{eq:projection-stability},
\eqref{eq:normal-correction-bound}, and
\eqref{eq:holdall-uniform-trace} with the preceding estimates gives
\[
\begin{aligned}
  \|\bm N_h^{nn,q}\|_{H(\DivDiv,\Omega_h)}
  &\lesssim
  \|\Pi_h^{nn}\bm N_h\|_{H(\DivDiv,\Omega_h)}\\
  &\lesssim
  \|\bm N_h\|_{H(\DivDiv,\Omega_h)}
  +\Bigl(\sum_{e\in\mathcal E_h^\partial}h_e|\mu_e^0|^2\Bigr)^{1/2}
  \lesssim\|v_h\|_{L^2(\Omega_h)}.
\end{aligned}
\]
This proves \eqref{eq:stable-infsup-witness}.
\end{proof}

On the discrete kernel, the graph norm reduces to the \(L^2\) norm.
By Brezzi's theory \cite{BoffiBrezziFortin2013},
the uniform coercivity of \(\mathbb A_\sigma\) and
Theorem~\ref{thm:infsup} imply that \eqref{eq:discrete-method} has
a unique solution for every \(f_h\in V_h\), satisfying $h$-uniform stability
\(\|\bm M_h\|_{H(\DivDiv,\Omega_h)}+\|u_h\|_{0,\Omega_h}
\lesssim\|f_h\|_{0,\Omega_h}\).

\subsection{Constrained interpolation}
\label{subsec:geometric-estimates}

For \(\bm N\in H^2(\widetilde\Omega;\mathbb S)\), define
\begin{equation}
 \label{eq:corrected-interpolant}
 \Pi_h\bm N
 :=\Pi_h^{nn,q}\Pi_h^{nn}
       \Pi_h^{\mathrm{dDiv}}(\bm N|_{\Omega_h}).
\end{equation}
For \(k=0,1\), let \(P_e^k\) be the \(L^2(e)\)-projection onto \(\mathbb{P}_k(e)\).

\begin{proposition}[constrained interpolation]
\label{prop:constrained-commuting-approximation}
Suppose Assumption~\ref{ass:polygonal-boundary} holds, each
\(\Gamma^{(i)}\) is piecewise \(C^{1,1}\), and
\(\mathcal V_\Gamma\subset\mathcal V_h^\partial\).
If \(\bm N\in H^2(\widetilde\Omega;\mathbb S)\) satisfies
\(N_{nn}=0\) on every \(\zeta\in\mathcal C_\Gamma\), then the operator $\Pi_h$ in
\eqref{eq:corrected-interpolant} satisfies \(\Pi_h\bm N\in\Sigma_h^{nn,q}\) and
\begin{subequations}
\begin{align}
 \DivDiv \Pi_h\bm N&=P_h^1\DivDiv\bm N,
 \label{eq:dual-interpolant-commuting}\\
 \|\bm N-\Pi_h\bm N\|_{0,\Omega_h}
 &\lesssim h^{3/2}\|\bm N\|_{H^2(\widetilde\Omega)},
 \label{eq:dual-interpolant-approximation}\\
 q_n(\Pi_h\bm N)|_e&=(I-P_e^0)P_e^1q_n(\bm N)
 \quad\forall e\in\mathcal E_h^\partial,
 \label{eq:corrected-shear-projection}\\
 \|q_n(\Pi_h\bm N)\|_{0,\Gamma_h}
 &\lesssim\|\bm N\|_{H^2(\widetilde\Omega)}.
 \label{eq:dual-interpolant-shear}
\end{align}
\end{subequations}
\end{proposition}

\begin{proof}
First, \eqref{eq:dual-interpolant-commuting} follows from
\eqref{eq:holdall-full-correction} and \eqref{eq:canonical-commuting}.
Let \(\mu_e^0,\mu_e^1,\mu_e^q\), \(e\in\mathcal E_h^\partial\), be the
boundary coefficients of \(\Pi_h^{\mathrm{dDiv}}\bm N\) in
\eqref{eq:holdall-three-data}.

{\it Step 1: normal--normal mean.}
We consider \(N_{n_en_e}\), \(N_{n_et_e}\), and \(N_{t_et_e}\)
in the fixed frame \((\bm n_e,\bm t_e)\), including on \(\Gamma_e\).
By \eqref{eq:boundary-frame-formulas}, the condition \(N_{nn}=0\) on
\(\Gamma_e\) gives
\[
\begin{aligned}
 N_{n_en_e}|_e
   &=2\rho_e'(N_{n_et_e}\circ\Phi_e)
       -|\rho_e'|^2(N_{t_et_e}\circ\Phi_e)+\delta_e,\\
 \delta_e&:=N_{n_en_e}|_e-N_{n_en_e}\circ\Phi_e.
\end{aligned}
\]
By \eqref{eq:boundary-trace-transfer},
\(\|\delta_e\|_{0,e}^2\lesssim
h_e^2\|\nabla\bm N\|_{0,\mathcal S_e}^2\).
Since \(\mu_e^0=h_e^{-1}\int_e N_{n_en_e}\) and
\(\rho_e|_{\partial e}=0\), integration by parts gives
\[
 \int_e\rho_e'(N_{n_et_e}\circ\Phi_e)\,ds
 =-\int_e\rho_e\,\partial_{t_e}(N_{n_et_e}\circ\Phi_e)\,ds.
\]
The bounds on \(\rho_e\) and \(\rho_e'\), together with
\eqref{eq:boundary-pullback-estimates}, imply
\begin{equation*}
 |\mu_e^0|^2
 \lesssim h_e^3\|\bm N\|_{H^1(\Gamma_e)}^2
             +h_e\|\nabla\bm N\|_{0,\mathcal S_e}^2.
\end{equation*}
Applying \eqref{eq:boundary-strip-estimate} to \(\nabla\bm N\) in
the above estimate, and using trace bound
\eqref{eq:uniform-trace}, gives
\begin{equation}  \label{eq:normal-mean-geometric-bound}
 \sum_{e\in\mathcal E_h^\partial}h_e|\mu_e^0|^2
 \lesssim h^4 \|\bm N\|_{H^1(\Gamma_h)}^2
 + h^4 \|\nabla \bm N\|_{0,\Gamma_h}^2
 + h^6 \|\bm N\|_{H^2(\mathcal U)}^2
 \lesssim h^4 \|\bm N\|_{H^2(\widetilde\Omega)}^2.
\end{equation}

\medskip\noindent
{\it Step 2: odd and shear modes.}
Preservation of the normal--normal moments gives
\(\mu_e^1=(N_{n_en_e},\ell_e)_e/\|\ell_e\|_{0,e}^2\), where
\(\|\ell_e\|_{0,e}^2=\int_0^{h_e}(2s/h_e-1)^2\,ds=h_e/3\).
By Cauchy--Schwarz, the decomposition in Step 1, and the bounds
\(\|\rho_e'\|_{L^\infty(e)}\lesssim h_e\),
\eqref{eq:boundary-pullback-estimates}, \eqref{eq:uniform-trace}, and
\eqref{eq:boundary-trace-transfer}, we obtain
\[
\begin{aligned}
 \sum_{e\in\mathcal E_h^\partial} h_e^2|\mu_e^1|^2
 &\le 3\sum_{e\in\mathcal E_h^\partial} h_e\|N_{n_en_e}\|_{0,e}^2\\
 &\lesssim\sum_{e\in\mathcal E_h^\partial} h_e^3
   \bigl(\|\bm N\|_{0,\Gamma_e}^2
          +\|\nabla\bm N\|_{0,\mathcal S_e}^2\bigr)\\
 &\lesssim h^3\bigl(\|\bm N\|_{0,\Gamma}^2
          +\|\nabla\bm N\|_{0,\mathcal U}^2\bigr)
 \lesssim h^3\|\bm N\|_{H^2(\widetilde\Omega)}^2.
\end{aligned}
\]

Also \(\mu_e^q=\int_e q_n(\bm N)\). Since \(q_n(\bm N)\) on a
straight edge is a linear combination of first derivatives,
Cauchy--Schwarz and \eqref{eq:uniform-trace} give
\[
 \sum_{e\in\mathcal E_h^\partial} h_e^2|\mu_e^q|^2
 \le\sum_{e\in\mathcal E_h^\partial} h_e^3\|q_n(\bm N)\|_{0,e}^2
 \lesssim h^3\|\bm N\|_{H^2(\widetilde\Omega)}^2.
\]
Together with \eqref{eq:normal-mean-geometric-bound}, these estimates yield
\begin{equation}
 \sum_{e\in\mathcal E_h^\partial}
 \bigl(h_e|\mu_e^0|^2
       +h_e^2|\mu_e^1|^2
       +h_e^2|\mu_e^q|^2\bigr)
 \lesssim h^3\|\bm N\|_{H^2(\widetilde\Omega)}^2.
 \label{eq:compatible-trace-defect}
\end{equation}
The correction bound \eqref{eq:holdall-correction-bound},
\eqref{eq:compatible-trace-defect}, and
\eqref{eq:canonical-approximation} prove
\eqref{eq:dual-interpolant-approximation}.

\medskip\noindent
{\it Step 3: shear trace.}
The canonical interpolant has shear trace \(P_e^1q_n(\bm N)\).
The normal--normal corrections in Section~\ref{sec:normal-normal-correction}
leave this trace unchanged, and the effective-shear correction subtracts
its mean on every boundary edge. This proves
\eqref{eq:corrected-shear-projection}.
Projection stability and \eqref{eq:uniform-trace} then give
\eqref{eq:dual-interpolant-shear}.
\end{proof}

\subsection{Consistency and moment error estimate}
\label{subsec:moment-error}

Throughout this subsection, we assume that each \(\Gamma^{(i)}\) is
piecewise \(C^{1,1}\), the meshes satisfy
Assumption~\ref{ass:polygonal-boundary}, and
\(\mathcal V_\Gamma\subset\mathcal V_h^\partial\).
Let \(u\in H^4(\Omega)\) be the exact displacement and choose a bounded
extension \(\widetilde u\in H^4(\widetilde\Omega)\), with
\(\|\widetilde u\|_{H^4(\widetilde\Omega)}
\lesssim\|u\|_{H^4(\Omega)}\)
\cite[Theorem~1.4.3.1]{Grisvard1985}. Set
\begin{equation}
 \widetilde{\bm M}:=\mathbb C_\sigma D^2\widetilde u,
 \qquad
 r_{\mathrm{ext}}:=\DivDiv\widetilde{\bm M}-\widetilde f.
 \label{eq:compatible-moment-extension}
\end{equation}
Then \(\widetilde{\bm M}|_\Omega=\bm M\),
\(\widetilde{\bm M}\in H^2(\widetilde\Omega;\mathbb S)\), and
\(r_{\mathrm{ext}}=0\) in \(\Omega\).
Recall that the load extension \(\widetilde f\) is prescribed in
Section~\ref{subsec:boundary-spaces-method}.

For a pair \((\bm N,v)\in H(\DivDiv,\Omega_h;\mathbb S)\times L^2(\Omega_h)\),
define the two residuals by
\begin{align}
 R_h^{\mathrm{bd}}(\bm N,v;\bm Q_h)
 &:=(\mathbb A_\sigma\bm N,\bm Q_h)_{\Omega_h}
       -(v,\DivDiv\bm Q_h)_{\Omega_h},
 \label{eq:boundary-residual}\\
 R_h^{\mathrm{vol}}(\bm N;w_h)
 &:=(\DivDiv\bm N-f_h,w_h)_{\Omega_h}.
 \label{eq:volume-residual}
\end{align}
For the extended exact solution \((\widetilde{\bm M},\widetilde u)\),
we abbreviate the discrete dual norms as
\begin{equation}
 \label{eq:residual-dual-norms}
 \|R_h^{\mathrm{bd}}\|_{(\Sigma_h^{nn,q})'}
 :=\sup_{\bm Q_h\ne0}
   \frac{|R_h^{\mathrm{bd}}(\widetilde{\bm M},\widetilde u;\bm Q_h)|}
        {\|\bm Q_h\|_{H(\DivDiv,\Omega_h)}},
 \quad
 \|R_h^{\mathrm{vol}}\|_{V_h'}
 :=\sup_{w_h\ne0}
   \frac{|R_h^{\mathrm{vol}}(\widetilde{\bm M};w_h)|}{\|w_h\|_{0,\Omega_h}},
\end{equation}
where the suprema run over \(\Sigma_h^{nn,q}\) and \(V_h\), respectively.

\paragraph{Boundary traces}
On every boundary edge, \(q_n(\bm Q_h)\) is linear with zero mean for
\(\bm Q_h\in\Sigma_h^{nn,q}\), so it pairs only with the odd part
of the displacement trace. We define \(w_{\mathrm o}\) edgewise on
\(\Gamma_h\) by
\(w_{\mathrm o}(\bm x):=(w(\bm x)-w(\bm z_0+\bm z_1-\bm x))/2\)
for \(\bm x\in e=[\bm z_0,\bm z_1]\).

\begin{lemma}[odd boundary trace]
\label{lem:H4-odd-chord}
If \(w\in H^4(\widetilde\Omega)\) vanishes on \(\Gamma\), then
\begin{equation}
 \|w_{\mathrm o}\|_{0,\Gamma_h}
 \lesssim h^3\|w\|_{H^4(\widetilde\Omega)}.
 \label{eq:H4-odd-chord}
\end{equation}
\end{lemma}

\begin{proof}
Since \(w_{\mathrm o}\) vanishes at the endpoints and midpoint of \(e\),
its quadratic interpolant at these nodes is zero.
The Bramble--Hilbert lemma and scaling argument therefore give $\|w_{\mathrm o}\|_{0,e}
 \lesssim h_e^3|w_{\mathrm o}|_{H^3(e)}
 \lesssim h_e^3\|\partial_{t_e}^3w\|_{0,e}$.
Summing over the boundary edges and applying \eqref{eq:uniform-trace} to \(D^3w\)
proves the assertion.
\end{proof}

Now, the elementwise Green identity gives, for \(\bm Q_h\in\Sigma_h^{nn,q}\),
\begin{equation}
\begin{aligned}
 R_h^{\mathrm{bd}}(\widetilde{\bm M},\widetilde u;\bm Q_h)
 &=(D^2\widetilde u,\bm Q_h)_{\Omega_h}
       -(\widetilde u,\DivDiv\bm Q_h)_{\Omega_h}\\
 &=-\sum_{e\in\mathcal E_h^\partial}\int_e
       \widetilde u_{\mathrm o}q_n(\bm Q_h)\,ds.
\end{aligned}
 \label{eq:green-residual-odd-form}
\end{equation}
Here the interior traces cancel by conformity,
\(Q_{h,nn}=0\), and the corner-force terms vanish because the
boundary vertices lie on \(\Gamma\). If an edge coincides with an exact
straight part of \(\Gamma\), then \(\widetilde u|_e=0\), and its
contribution vanishes.

\paragraph{Consistency estimates}
Set \(\mathcal S_h^+:=\Omega_h\setminus\Omega\).
Since \(r_{\mathrm{ext}}=0\) in \(\Omega\), the load projection
in \eqref{eq:discrete-load} gives
\(R_h^{\mathrm{vol}}(\widetilde{\bm M};w_h)
=(r_{\mathrm{ext}},w_h)_{\mathcal S_h^+}\) for \(w_h\in V_h\).
\begin{proposition}[boundary and load consistency]
\label{thm:boundary-consistency}
For every \(\bm Q_h\in\Sigma_h^{nn,q}\),
\begin{equation}
 |R_h^{\mathrm{bd}}(\widetilde{\bm M},\widetilde u;\bm Q_h)|
 \lesssim h^{3/2}\|\widetilde u\|_{H^4(\widetilde\Omega)}
                    \|\bm Q_h\|_{0,\Omega_h}.
 \label{eq:boundary-consistency}
\end{equation}
If \(\Omega_h\subset\Omega\), then \(R_h^{\mathrm{vol}}=0\).
Otherwise, under the additional assumption
\(r_{\mathrm{ext}}\in L^\infty(\mathcal U)\), 
\begin{equation}
 \|R_h^{\mathrm{vol}}\|_{V_h'}
 \lesssim h^{3/2}\|r_{\mathrm{ext}}\|_{L^\infty(\mathcal U)}.
 \label{eq:volume-consistency}
\end{equation}
\end{proposition}

\begin{proof}
A standard scaling argument gives
\(\|q_n(\bm Q_h)\|_{0,e}\lesssim h_e^{-3/2}\|\bm Q_h\|_{0,K_e}\).
Together with \eqref{eq:green-residual-odd-form},
\eqref{eq:H4-odd-chord}, and quasi-uniformity, this yields
\[
 |R_h^{\mathrm{bd}}(\widetilde{\bm M},\widetilde u;\bm Q_h)|
 \le\|\widetilde u_{\mathrm o}\|_{0,\Gamma_h}
             \|q_n(\bm Q_h)\|_{0,\Gamma_h}
 \lesssim h^{3/2}\|\widetilde u\|_{H^4(\widetilde\Omega)}
                         \|\bm Q_h\|_{0,\Omega_h}.
\]
The boundary-distance estimate gives \(|K\cap\mathcal S_h^+|\lesssim h^3\),
and only \(\mathcal O(h^{-1})\) elements intersect \(\mathcal S_h^+\).
The inverse inequality and Cauchy--Schwarz yield, for \(w_h\in V_h\),
\[
\begin{aligned}
 \|w_h\|_{L^1(\mathcal S_h^+)}
 &\le\sum_{K\cap\mathcal S_h^+\ne\varnothing}
       |K\cap\mathcal S_h^+|\,\|w_h\|_{L^\infty(K)}\\
 &\lesssim h^2\sum_{K\cap\mathcal S_h^+\ne\varnothing}\|w_h\|_{0,K}
 \lesssim h^{3/2}\|w_h\|_{0,\Omega_h}.
\end{aligned}
\]
Thus
\(|R_h^{\mathrm{vol}}(\widetilde{\bm M};w_h)|
\le\|r_{\mathrm{ext}}\|_{L^\infty(\mathcal U)}
\|w_h\|_{L^1(\mathcal S_h^+)}\), which proves \eqref{eq:volume-consistency}.
\end{proof}

\begin{remark}[role of the shear-mean constraint]
\label{rem:role-shear-mean-constraint}
The improvement in \eqref{eq:boundary-consistency} comes from testing
only the odd trace on polygonal approximations of curved boundary
pieces. Without the shear-mean constraint, the full-trace bound
\(\|\widetilde u\|_{0,\Gamma_h}\lesssim
h^2\|\widetilde u\|_{H^3(\widetilde\Omega)}\) gives only
\(h^{1/2}\) consistency. The constraint is imposed uniformly on all
boundary edges; exact straight edges contribute no boundary consistency
error.
\end{remark}

\begin{corollary}[moment and displacement errors]
\label{cor:total-error}
The discrete solution of \eqref{eq:discrete-method} satisfies
\begin{equation}
 \|\widetilde{\bm M}-\bm M_h\|_{0,\Omega_h}
 +\|\widetilde u-u_h\|_{0,\Omega_h}
 \lesssim h^{3/2}\|u\|_{H^4(\Omega)}
             +\|R_h^{\mathrm{vol}}\|_{V_h'}.
 \label{eq:total-error}
\end{equation}
The errors are of order \(h^{3/2}\) for \(\Omega_h\subset\Omega\),
and also for general approximations with
\(r_{\mathrm{ext}}\in L^\infty(\mathcal U)\) by
\eqref{eq:volume-consistency}.
\end{corollary}

\begin{proof}
Set \(\bm\Xi_h:=\Pi_h\widetilde{\bm M}-\bm M_h\) and
\(e_h:=P_h^1\widetilde u-u_h\).
The definitions \eqref{eq:boundary-residual}--\eqref{eq:volume-residual},
projection orthogonality, and \eqref{eq:dual-interpolant-commuting} give
the error equations
\begin{equation}
\label{eq:discrete-error-equations}
\begin{aligned}
 (\mathbb A_\sigma\bm\Xi_h,\bm Q_h)_{\Omega_h}
 -(e_h,\DivDiv\bm Q_h)_{\Omega_h}
 &=(\mathbb A_\sigma(\Pi_h\widetilde{\bm M}-\widetilde{\bm M}),
       \bm Q_h)_{\Omega_h}\\
 &\quad+R_h^{\mathrm{bd}}(\widetilde{\bm M},\widetilde u;\bm Q_h),\\
 (\DivDiv\bm\Xi_h,w_h)_{\Omega_h}
 &=R_h^{\mathrm{vol}}(\widetilde{\bm M};w_h),
\end{aligned}
\end{equation}
for \((\bm Q_h,w_h)\in\Sigma_h^{nn,q}\times V_h\).
Brezzi's theory and \eqref{eq:residual-dual-norms} yield
\[
\begin{aligned}
 \|\bm\Xi_h\|_{H(\DivDiv,\Omega_h)}+\|e_h\|_{0,\Omega_h}
 \lesssim\|\Pi_h\widetilde{\bm M}-\widetilde{\bm M}\|_{0,\Omega_h}
 +\|R_h^{\mathrm{bd}}\|_{(\Sigma_h^{nn,q})'}
       +\|R_h^{\mathrm{vol}}\|_{V_h'}.
\end{aligned}
\]
The estimates \eqref{eq:dual-interpolant-approximation} and
\eqref{eq:boundary-consistency}, followed by the triangle inequality,
prove \eqref{eq:total-error}.
\end{proof}

\begin{remark}[polygonal domains]
The estimate \eqref{eq:total-error} also applies when \(\Omega\) is
polygonal, since the effective-shear mean constraint is imposed on every
boundary edge. In that special case the constraint is unnecessary for
consistency, and omitting it can retain the second-order approximation of
the underlying element. We impose it uniformly so that the method does
not require a classification of the physical boundary.
\end{remark}

\subsection{Displacement \texorpdfstring{\(L^2\)}{L2} error estimate}
\label{subsec:displacement-L2-error}

Under the \(H^4\)-regularity assumption stated below, we use a duality
argument to derive a refined displacement estimate. The argument requires
neither \(\Omega_h\subset\Omega\) nor
\(r_{\mathrm{ext}}\in L^\infty(\mathcal U)\).
We retain \(u\in H^4(\Omega)\), the compatible extension
\eqref{eq:compatible-moment-extension}, and the prescribed load
\(\widetilde f\in L^2(\widetilde\Omega)\) from \eqref{eq:discrete-load}. Write
\(\bm E_h:=\widetilde{\bm M}-\bm M_h\), and recall the projected errors
\begin{equation*}
 e_h=P_h^1\widetilde u-u_h, \qquad \bm\Xi_h=\Pi_h\widetilde{\bm M}-\bm M_h.
\end{equation*}

For \(w_h\in V_h\), a scaling argument and
\(|K\cap\mathcal S_h^+|\lesssim h^3\) give
\begin{equation}
 \|w_h\|_{0,\mathcal S_h^+}^2
 \le\sum_{K\in\mathcal T_h}|K\cap\mathcal S_h^+|
                     \|w_h\|_{L^\infty(K)}^2
 \lesssim h\|w_h\|_{0,\Omega_h}^2.
 \label{eq:discrete-outer-strip}
\end{equation}
which leads to 
 $\|R_h^{\mathrm{vol}}\|_{V_h'}
 \lesssim h^{1/2}\|r_{\mathrm{ext}}\|_{0,\mathcal S_h^+}
 \lesssim h^{1/2}
 (\|u\|_{4,\Omega}+\|\widetilde f\|_{0,\widetilde\Omega})$.
Therefore, \eqref{eq:total-error} and
\eqref{eq:dual-interpolant-approximation} yield a preliminary estimate
\begin{equation}
 \|\bm E_h\|_{0,\Omega_h}+\|\bm\Xi_h\|_{0,\Omega_h}
 \lesssim h^{1/2}
       \bigl(\|u\|_{4,\Omega}+\|\widetilde f\|_{0,\widetilde\Omega}\bigr).
 \label{eq:preliminary-L2-moment-bound}
\end{equation}
The stronger \(h^{3/2}\) moment rate is not asserted under these weaker
data assumptions, but \eqref{eq:preliminary-L2-moment-bound} suffices for
the displacement estimate below.

\paragraph{Duality argument}
We introduce the auxiliary plate problem on $\Omega$:
\begin{equation}
 \bm Z=\mathbb C_\sigma D^2z,\quad
 \DivDiv\bm Z=g\ \text{ in }\Omega,\qquad
 z=0,\quad Z_{nn}=0\ \text{ on }\Gamma.
 \label{eq:fixed-domain-dual-problem}
\end{equation}

\begin{assumption}[$H^4$-regularity]
\label{ass:fixed-domain-dual-regularity}
For \(g\in L^2(\Omega)\), the solution of
\eqref{eq:fixed-domain-dual-problem} satisfies
\begin{equation}
 \|z\|_{4,\Omega}+\|\bm Z\|_{2,\Omega}
 \lesssim\|g\|_{0,\Omega}.
 \label{eq:fixed-domain-dual-regularity}
\end{equation}
\end{assumption}
This assumption holds, for example, if every boundary component of
\(\Omega\) is of class \(C^4\), a stronger requirement than that used for
the geometric error estimates.

Let
\(e_h^0\) be the zero extension to \(\Omega\) of
\(e_h|_{\Omega\cap\Omega_h}\), and take \(g=e_h^0\) in
\eqref{eq:fixed-domain-dual-problem}.
Choose a bounded \(H^4\) extension \(\widetilde z\) and set
\(\widetilde{\bm Z}:=\mathbb C_\sigma D^2\widetilde z\).
Then \eqref{eq:fixed-domain-dual-regularity} gives
\begin{equation}
 \|\widetilde z\|_{4,\widetilde\Omega}
 +\|\widetilde{\bm Z}\|_{2,\widetilde\Omega}
 \lesssim\|z\|_{4,\Omega}
 \lesssim\|e_h\|_{0,\Omega_h}.
 \label{eq:extended-dual-regularity}
\end{equation}

\begin{proposition}[perturbed duality estimate]
\label{prop:dual-error-identity}
Under Assumption~\ref{ass:fixed-domain-dual-regularity}, one has
\begin{subequations}
\begin{align}
 (1-Ch^{1/2})\|e_h\|_{0,\Omega_h}^2
 &\le(e_h,\DivDiv\Pi_h\widetilde{\bm Z})_{\Omega_h},
 \label{eq:dual-coercivity}\\
 \bigl|(e_h,\DivDiv\Pi_h\widetilde{\bm Z})_{\Omega_h}\bigr|
 &\lesssim h^2
       \bigl(\|u\|_{4,\Omega}+\|\widetilde f\|_{0,\widetilde\Omega}\bigr)
       \|z\|_{4,\Omega}.
 \label{eq:dual-pairing-bound}
\end{align}
\end{subequations}
\end{proposition}

\begin{proof}
{\it Step 1: exterior perturbation.}
Commutativity \eqref{eq:dual-interpolant-commuting} and the auxiliary problem \eqref{eq:fixed-domain-dual-problem} give
\[
\begin{aligned}
 (e_h,\DivDiv\Pi_h\widetilde{\bm Z})_{\Omega_h}
 &= \|e_h\|_{0,\Omega\cap\Omega_h}^2
       +(e_h,\DivDiv\widetilde{\bm Z})_{\mathcal S_h^+}  \\
  &= \|e_h\|_{0,\Omega_h}^2-\|e_h\|_{0,\mathcal S_h^+}^2
       +(e_h,\DivDiv\widetilde{\bm Z})_{\mathcal S_h^+}.
\end{aligned}
\]
By \eqref{eq:discrete-outer-strip} and
\eqref{eq:extended-dual-regularity}, the last two terms have bound \(Ch^{1/2}\|e_h\|_{0,\Omega_h}^2\), leading to \eqref{eq:dual-coercivity}.

{\it Step 2: error identity and geometric estimate.}
Test the first equation in \eqref{eq:discrete-error-equations}
with \(\Pi_h\widetilde{\bm Z}\).
Using \(\mathbb A_\sigma\widetilde{\bm Z}=D^2\widetilde z\),
\(\DivDiv\bm M_h=f_h\), and
\(\bm M_h=\Pi_h\widetilde{\bm M}-\bm\Xi_h\), we obtain
\[
\begin{aligned}
 (e_h,\DivDiv\Pi_h\widetilde{\bm Z})_{\Omega_h}
 ={}&(\mathbb A_\sigma\bm E_h,\Pi_h\widetilde{\bm Z})_{\Omega_h}
       -R_h^{\mathrm{bd}}(\widetilde{\bm M},\widetilde u;
                          \Pi_h\widetilde{\bm Z})\\
 ={}&(\mathbb A_\sigma\bm E_h,
             \Pi_h\widetilde{\bm Z}-\widetilde{\bm Z})_{\Omega_h}
       +(D^2\widetilde z,\widetilde{\bm M})_{\Omega_h}
       -(D^2\widetilde z,\bm M_h)_{\Omega_h}\\
 &-R_h^{\mathrm{bd}}(\widetilde{\bm M},\widetilde u;
                     \Pi_h\widetilde{\bm Z})\\
 ={}&(\mathbb A_\sigma\bm E_h,
             \Pi_h\widetilde{\bm Z}-\widetilde{\bm Z})_{\Omega_h}
       +\mathcal G_h(u,z)
       -R_h^{\mathrm{bd}}(\widetilde{\bm Z},\widetilde z;\bm M_h)\\
 &-R_h^{\mathrm{bd}}(\widetilde{\bm M},\widetilde u;\Pi_h\widetilde{\bm Z})
       +(\widetilde z,\widetilde f-P_h^1\widetilde f)_{\Omega_h}.
\end{aligned}
\]
With \(\mathcal S_h^-:=\Omega\setminus\Omega_h\), the geometric term is
\[
\begin{aligned}
 \mathcal G_h(u,z)
 :=(D^2\widetilde z,\widetilde{\bm M})_{\Omega_h}
                        -(\widetilde z,\widetilde f)_{\Omega_h}
 =\left(\int_{\mathcal S_h^+}-\int_{\mathcal S_h^-}\right)
       \bigl(D^2\widetilde z:\widetilde{\bm M}
                      -\widetilde z\,\widetilde f\bigr)\,dx,
\end{aligned}
\]
where the second equality follows from
\((D^2z,\bm M)_\Omega=(z,f)_\Omega\).
The symmetric difference has area \(\mathcal O(h^2)\) and lies
within distance \(\mathcal O(h^2)\) of \(\Gamma\).
Thus \(\widetilde z|_\Gamma=0\) and Sobolev embedding give
\(\|\widetilde z\|_{L^\infty(\Omega\triangle\Omega_h)}
\lesssim h^2\|\nabla\widetilde z\|_{L^\infty(\mathcal U)}
\lesssim h^2\|z\|_{4,\Omega}\). 
Bounding \(D^2\widetilde z\) and \(\widetilde{\bm M}\) in
\(L^\infty\), and the load term by Cauchy--Schwarz, yields
\begin{equation}
 |\mathcal G_h(u,z)|
 \lesssim
 \bigl(h^2\|u\|_{4,\Omega}
           +h^3\|\widetilde f\|_{0,\widetilde\Omega}\bigr)\|z\|_{4,\Omega}.
 \label{eq:leading-geometric-defect-bound}
\end{equation}
{\it Step 3: remaining terms.}
Using \(\bm M_h=\Pi_h\widetilde{\bm M}-\bm\Xi_h\), the approximation
estimate \eqref{eq:dual-interpolant-approximation}, the odd-trace identity
\eqref{eq:green-residual-odd-form}, the estimate
\eqref{eq:H4-odd-chord}, and a scaling argument imply
\[
 \bigl|(\mathbb A_\sigma\bm E_h,
             \Pi_h\widetilde{\bm Z}-\widetilde{\bm Z})_{\Omega_h}\bigr|
       +|R_h^{\mathrm{bd}}(\widetilde{\bm Z},\widetilde z;\bm\Xi_h)|
 \lesssim h^{3/2}
       \bigl(\|\bm E_h\|_{0,\Omega_h}+\|\bm\Xi_h\|_{0,\Omega_h}\bigr)
       \|z\|_{4,\Omega}.
\]
By \eqref{eq:preliminary-L2-moment-bound}, this is bounded by the
right-hand side of \eqref{eq:dual-pairing-bound}.
Using the stable shear trace \eqref{eq:dual-interpolant-shear}
instead of scaling gives
\[
 |R_h^{\mathrm{bd}}(\widetilde{\bm Z},\widetilde z;\Pi_h\widetilde{\bm M})|
       +|R_h^{\mathrm{bd}}(\widetilde{\bm M},\widetilde u;\Pi_h\widetilde{\bm Z})|
 \lesssim h^3\|u\|_{4,\Omega}\|z\|_{4,\Omega}.
\]
Finally, projection orthogonality and local approximation yield
\[
\begin{aligned}
 |(\widetilde z,\widetilde f-P_h^1\widetilde f)_{\Omega_h}|
 =|(\widetilde z-P_h^1\widetilde z,\widetilde f)_{\Omega_h}|
 \lesssim h^2\|\widetilde z\|_{2,\widetilde\Omega}
                   \|\widetilde f\|_{0,\widetilde\Omega}
 \lesssim h^2\|z\|_{4,\Omega}\|\widetilde f\|_{0,\widetilde\Omega}.
\end{aligned}
\]
Together with \eqref{eq:leading-geometric-defect-bound}, these estimates
prove \eqref{eq:dual-pairing-bound}.
\end{proof}

Combining \eqref{eq:dual-coercivity}, \eqref{eq:dual-pairing-bound}, and
\eqref{eq:extended-dual-regularity}, and taking \(h\) sufficiently small,
leads to the second-order error estimate for $\|e_h\|_{0,\Omega_h}$. The triangle inequality and the local
\(\mathbb P_1\)-projection estimate then yield the following result.

\begin{theorem}[displacement \(L^2\)-error estimate]
\label{thm:second-order-displacement-L2-error}
Suppose Assumptions~\ref{ass:polygonal-boundary} and
\ref{ass:fixed-domain-dual-regularity} hold, each \(\Gamma^{(i)}\)
is piecewise \(C^{1,1}\), and
\(\mathcal V_\Gamma\subset\mathcal V_h^\partial\).
Let \(u\in H^4(\Omega)\), let \(\widetilde u\) be the bounded extension in
\eqref{eq:compatible-moment-extension}, and let
\(\widetilde f\in L^2(\widetilde\Omega)\).
Then, for sufficiently small \(h\),
\[
 \|P_h^1\widetilde u-u_h\|_{0,\Omega_h}
       +\|\widetilde u-u_h\|_{0,\Omega_h}
 \lesssim h^2
       \bigl(\|u\|_{4,\Omega}+\|\widetilde f\|_{0,\widetilde\Omega}\bigr).
\]
\end{theorem}
 \section{Local cubic postprocessing}
\label{sec:supercloseness-postprocessing}

We adapt the standard local cubic postprocessing of
F\"uhrer and Heuer \cite[Section~4.1]{FuehrerHeuer2025} to the corrected
solution and the material tensor \(\mathbb C_\sigma\). It transfers the
moment estimate to a broken Hessian estimate, while the affine component
is controlled by the projected displacement error.

For \(K\in\mathcal T_h\), let \(P_K^1\) be the \(L^2(K)\)-projection onto
\(\mathbb P_1(K)\), and set
\[
 Z_3(K):=\{w\in\mathbb P_3(K):P_K^1w=0\},
 \quad
 a_K(v,w):=(\mathbb C_\sigma D^2v,D^2w)_K.
\]
We define \(u_h^\star|_K\in\mathbb P_3(K)\) by
\begin{equation}
\label{eq:local-cubic-postprocessing}
\begin{aligned}
 a_K(u_h^\star,w)&=(\bm M_h,D^2w)_K
       \quad\text{for all }w\in Z_3(K),\\
 P_K^1u_h^\star&=u_h|_K.
\end{aligned}
\end{equation}
Since \(\ker(D^2|_{\mathbb P_3(K)})=\mathbb P_1(K)\), ellipticity of
\(a_K\) gives a unique solution on every element. The reconstruction is
piecewise cubic and generally discontinuous. We write \(D_h^2\) for the
elementwise Hessian.

\begin{proposition}[postprocessing error decomposition]
\label{prop:local-postprocessing-errors}
Let \(e_h=P_h^1\widetilde u-u_h\).
For \(\widetilde u\) and \(\widetilde{\bm M}\) in
\eqref{eq:compatible-moment-extension}, we have
\begin{align}
 \|D_h^2(\widetilde u-u_h^\star)\|_{0,\Omega_h}
 &\lesssim h^2\|\widetilde u\|_{4,\widetilde\Omega}
       +\|\widetilde{\bm M}-\bm M_h\|_{0,\Omega_h},
 \label{eq:postprocessing-hessian-error}\\
 \|\widetilde u-u_h^\star\|_{0,\Omega_h}
 &\lesssim h^2\|D_h^2(\widetilde u-u_h^\star)\|_{0,\Omega_h}
       +\|e_h\|_{0,\Omega_h}.
 \label{eq:postprocessing-displacement-error}
\end{align}
\end{proposition}

\begin{proof}
For each \(K\), define \(S_K\widetilde u\in\mathbb P_3(K)\) by
\[
 a_K(S_K\widetilde u,w)=a_K(\widetilde u,w)
 \quad\forall w\in Z_3(K),\qquad
 P_K^1S_K\widetilde u=P_K^1\widetilde u.
\]
Polynomial approximation and ellipticity lead to $\|D^2(\widetilde u-S_K\widetilde u)\|_{0,K}
 \lesssim h_K^2|\widetilde u|_{4,K}$.
For \(\delta_K:=S_K\widetilde u-u_h^\star\), subtracting the two local
equations leads to \(a_K(\delta_K,w)=(\widetilde{\bm M}-\bm M_h,D^2w)_K\)
for all \(w\in Z_3(K)\). Since \(D^2P_K^1\delta_K=0\), the choice
\(w=(I-P_K^1)\delta_K\in Z_3(K)\) satisfies \(D^2w=D^2\delta_K\).
Ellipticity and the Cauchy--Schwarz inequality therefore yield
\(\|D^2\delta_K\|_{0,K}\lesssim
\|\widetilde{\bm M}-\bm M_h\|_{0,K}\).
Combining this estimate with the approximation bound for
\(S_K\widetilde u\), and summing over \(K\), proves
\eqref{eq:postprocessing-hessian-error}.

For \(v:=\widetilde u-u_h^\star\), the constraint in
\eqref{eq:local-cubic-postprocessing} gives \(P_K^1v=e_h|_K\).
The local estimate
\(\|v-P_K^1v\|_{0,K}\lesssim h_K^2\|D^2v\|_{0,K}\) proves
\eqref{eq:postprocessing-displacement-error}.
\end{proof}

The preceding decomposition and the error estimates of
Section~\ref{sec:stability-error} give the following bounds.

\begin{corollary}[postprocessing error bounds]
\label{cor:local-postprocessing-rates}
Under the hypotheses of Corollary~\ref{cor:total-error}, the first
estimate below holds. If Assumption~\ref{ass:fixed-domain-dual-regularity}
also holds, then so does the second:
\begin{align*}
 \|D_h^2(\widetilde u-u_h^\star)\|_{0,\Omega_h}
 &\lesssim h^{3/2}\|u\|_{4,\Omega}
       +\|R_h^{\mathrm{vol}}\|_{V_h'},\\
 \|\widetilde u-u_h^\star\|_{0,\Omega_h}
 &\lesssim h^2
       \bigl(\|u\|_{4,\Omega}
             +\|\widetilde f\|_{0,\widetilde\Omega}\bigr).
\end{align*}
By \eqref{eq:volume-consistency}, the first estimate is of order
\(h^{3/2}\) for inner approximations, and also for general approximations
when \(r_{\mathrm{ext}}\in L^\infty(\mathcal U)\). The second requires
neither \(\Omega_h\subset\Omega\) nor
\(r_{\mathrm{ext}}\in L^\infty(\mathcal U)\).
\end{corollary}

\begin{proof}
The first estimate follows from
\eqref{eq:postprocessing-hessian-error} and
\eqref{eq:total-error}. Under the dual regularity assumption,
Theorem~\ref{thm:second-order-displacement-L2-error} bounds \(e_h\) by
order \(h^2\). Moreover, \eqref{eq:postprocessing-hessian-error} and
\eqref{eq:preliminary-L2-moment-bound} bound the Hessian error by order
\(h^{1/2}\) under the same data assumptions. Hence its contribution to
\eqref{eq:postprocessing-displacement-error} is of order \(h^{5/2}\),
which proves the second estimate. The stated \(h^{3/2}\) cases follow
from \eqref{eq:volume-consistency}.
\end{proof}
 \section{Numerical experiments}
\label{sec:numerical-experiments}

We consider three geometries. The disk is the original plate-paradox
example of Babu\v{s}ka and Pitk\"aranta~\cite{BabuskaPitkaranta1990} and
permits a direct comparison with the uncorrected method. The
trefoil is smooth and nonconvex, with polygonal boundary edges on both
sides of the physical boundary. The final example combines multiple
boundary components, curved and exact straight pieces, and resolved
geometric corners.

All computations use the F\"uhrer--Heuer element and the local cubic
postprocessing in \eqref{eq:local-cubic-postprocessing}. We take
\(\sigma=0\) for the disk and the multiply connected plate, and
\(\sigma=0.3\) for the trefoil. Each mesh is refined by dividing every
triangle into four children and projecting new vertices on curved
boundary pieces onto the corresponding exact curves.
The manufactured solutions are evaluated throughout
\(\Omega_h\), including the parts outside \(\Omega\). In this section,
\(\|\cdot\|_0\) denotes the \(L^2(\Omega_h)\)-norm. The estimated orders of
convergence (EoCs) are computed from consecutive mesh levels.

The nonhomogeneous boundary data are imposed as described in
Remark~\ref{rem:nonhomogeneous-data}, with \(g_{D,h}=I_hg_D\) denoting
the continuous piecewise affine trace interpolant and with zero
effective-shear mean imposed on every boundary edge.

\subsection{The unit disk}
\label{subsec:numerical-disk}

Let \(\Omega=\{(x,y):x^2+y^2<1\}\) and
\[
 u(x,y)=\frac{(x^2+y^2)^2-6(x^2+y^2)+5}{64},
 \qquad f=1.
\]
Then \(u=M_{nn}=0\) on \(\Gamma\). The initial mesh consists of eight
triangles joining the origin to equally spaced boundary vertices. Four
uniform refinements give \(2{,}048\) triangles.

The corrected results in Table~\ref{tab:num-disk} show second-order
convergence for the displacement and order \(3/2\) for both the bending
moment and the broken Hessian of the postprocessed displacement, in
agreement with the estimates of Sections~\ref{sec:stability-error} and
\ref{sec:supercloseness-postprocessing}.

\begin{table}[!htbp]
\centering
\caption{Unit disk: corrected and uncorrected (\(\mathrm{nc}\)) errors and
estimated orders of convergence.}
\label{tab:num-disk}
\scriptsize
\setlength{\tabcolsep}{1.2pt}
\renewcommand{\arraystretch}{1.06}
\begin{tabular*}{\textwidth}{@{\extracolsep{\fill}}*{13}{c}@{}}
\toprule
$h$ & \multicolumn{2}{c}{$\|u-u_h\|_0$}
& \multicolumn{2}{c}{$\|u-u_h^{\star}\|_0$}
& \multicolumn{2}{c}{$\|D_h^2(u-u_h^{\star})\|_0$}
& \multicolumn{2}{c}{$\|\bm M-\bm M_h\|_0$}
& \multicolumn{2}{c}{$\|u-u_h^{\mathrm{nc}}\|_0$}
& \multicolumn{2}{c}{$\|\bm M-\bm M_h^{\mathrm{nc}}\|_0$}\\
\cmidrule(lr){2-3}\cmidrule(lr){4-5}\cmidrule(lr){6-7}
\cmidrule(lr){8-9}\cmidrule(lr){10-11}\cmidrule(lr){12-13}
1.0000 & $1.37\mathrm{e}{-2}$ & \multicolumn{1}{c}{--} & $1.22\mathrm{e}{-2}$ & \multicolumn{1}{c}{--} & $9.97\mathrm{e}{-2}$ & \multicolumn{1}{c}{--} & $1.27\mathrm{e}{-1}$ & \multicolumn{1}{c}{--} & $3.44\mathrm{e}{-2}$ & \multicolumn{1}{c}{--} & $1.64\mathrm{e}{-1}$ & \multicolumn{1}{c}{--}\\
0.5711 & $2.66\mathrm{e}{-3}$ & 2.93 & $2.05\mathrm{e}{-3}$ & 3.18 & $3.76\mathrm{e}{-2}$ & 1.74 & $4.77\mathrm{e}{-2}$ & 1.74 & $1.52\mathrm{e}{-2}$ & 1.46 & $1.22\mathrm{e}{-1}$ & 0.53\\
0.3022 & $5.64\mathrm{e}{-4}$ & 2.44 & $3.59\mathrm{e}{-4}$ & 2.74 & $1.31\mathrm{e}{-2}$ & 1.66 & $1.69\mathrm{e}{-2}$ & 1.63 & $6.77\mathrm{e}{-3}$ & 1.27 & $8.96\mathrm{e}{-2}$ & 0.48\\
0.1551 & $1.30\mathrm{e}{-4}$ & 2.20 & $7.02\mathrm{e}{-5}$ & 2.44 & $4.55\mathrm{e}{-3}$ & 1.58 & $5.97\mathrm{e}{-3}$ & 1.56 & $3.14\mathrm{e}{-3}$ & 1.15 & $6.49\mathrm{e}{-2}$ & 0.48\\
0.0785 & $3.13\mathrm{e}{-5}$ & 2.09 & $1.52\mathrm{e}{-5}$ & 2.25 & $1.59\mathrm{e}{-3}$ & 1.54 & $2.10\mathrm{e}{-3}$ & 1.53 & $1.51\mathrm{e}{-3}$ & 1.08 & $4.65\mathrm{e}{-2}$ & 0.49\\
\bottomrule
\end{tabular*}
\end{table}
 
To isolate the effect of the correction, we also solve the same problem
in \(\Sigma_h^{nn}\times V_h\), leaving the effective-shear mean
unconstrained. Denote this solution by
\((\bm M_h^{\mathrm{nc}},u_h^{\mathrm{nc}})\). The final EoC of the moment
error is \(0.49\), compared with \(1.53\) for the corrected solution. This
agrees with the \(\mathcal O(h^{1/2})\) boundary-consistency bound in
Remark~\ref{rem:role-shear-mean-constraint}. Repeating the duality
argument of Section~\ref{subsec:displacement-L2-error} with this
half-order moment bound gives first-order displacement convergence, as
observed in the table.

\subsection{A nonconvex trefoil domain}
\label{subsec:numerical-nonconvex}

The second domain is the asymmetric three-leaf domain used by Arnold and
Walker~\cite{ArnoldWalker2020}, with boundary parametrization
\[
 x(t)=[1+0.4\cos(3t)]\cos t,
 \quad
 y(t)=[1+(0.4+0.22\sin t)\cos(3t)]\sin t,
 \quad 0\leq t\leq2\pi.
\]
We choose \(u(x,y)=\sin(2\pi x)\cos(2\pi y)\) and
\(f=\Delta^2u=64\pi^4u\), with boundary data induced by \(u\) and
\(\bm M=\mathbb C_\sigma D^2u\). The initial mesh has \(184\) triangles.
Four uniform refinements give \(47{,}104\) triangles. Some polygonal
boundary edges lie outside \(\Omega\).

Table~\ref{tab:num-trefoil} shows order \(3/2\) for the moment and the
postprocessed broken Hessian. Both displacement errors converge at least
quadratically over the reported levels. Figure~\ref{fig:num-trefoil-fields}
illustrates the geometric mismatch and the postprocessed solution.

\begin{table}[!htbp]
\centering
\caption{Trefoil domain: errors and estimated orders of convergence.}
\label{tab:num-trefoil}
\scriptsize\setlength{\tabcolsep}{3.0pt}
\renewcommand{\arraystretch}{1.06}
\begin{tabular}{@{}*{9}{c}@{}}
\toprule
$h$ & \multicolumn{2}{c}{$\|u-u_h\|_0$} & \multicolumn{2}{c}{$\|u-u_h^\star\|_0$} & \multicolumn{2}{c}{$\|D_h^2(u-u_h^\star)\|_0$} & \multicolumn{2}{c}{$\|\bm M-\bm M_h\|_0$}\\
\cmidrule(lr){2-3}\cmidrule(lr){4-5}\cmidrule(lr){6-7}\cmidrule(lr){8-9}
0.3325 & $1.01\mathrm{e}{-1}$ & \multicolumn{1}{c}{--} & $5.66\mathrm{e}{-2}$ & \multicolumn{1}{c}{--} & $1.33\mathrm{e}{+1}$ & \multicolumn{1}{c}{--} & $1.02\mathrm{e}{+1}$ & \multicolumn{1}{c}{--}\\
0.1759 & $2.31\mathrm{e}{-2}$ & 2.31 & $6.73\mathrm{e}{-3}$ & 3.35 & $4.28\mathrm{e}{+0}$ & 1.79 & $3.47\mathrm{e}{+0}$ & 1.69\\
0.0911 & $5.67\mathrm{e}{-3}$ & 2.14 & $8.35\mathrm{e}{-4}$ & 3.17 & $1.41\mathrm{e}{+0}$ & 1.69 & $1.19\mathrm{e}{+0}$ & 1.62\\
0.0466 & $1.41\mathrm{e}{-3}$ & 2.07 & $1.28\mathrm{e}{-4}$ & 2.79 & $4.76\mathrm{e}{-1}$ & 1.61 & $4.15\mathrm{e}{-1}$ & 1.57\\
0.0235 & $3.53\mathrm{e}{-4}$ & 2.03 & $2.74\mathrm{e}{-5}$ & 2.26 & $1.65\mathrm{e}{-1}$ & 1.56 & $1.45\mathrm{e}{-1}$ & 1.54\\
\bottomrule
\end{tabular}
\end{table}
 
\begin{figure}[!htbp]
 \centering
 \includegraphics[width=0.82\textwidth]{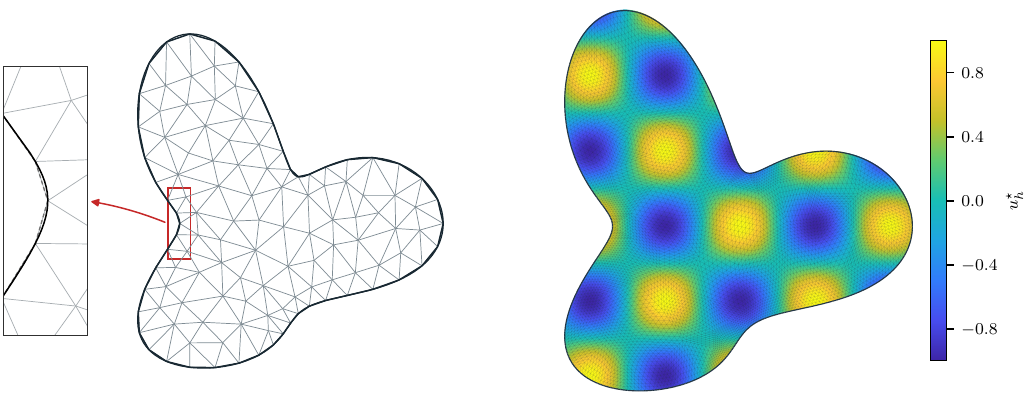}
 \caption{Trefoil domain. The left panel enlarges the red-framed concave
 patch of the initial mesh shown in the center. The solid curve is
 \(\Gamma\), and the dashed segments form \(\Gamma_h\). The right panel
 shows \(u_h^\star\) on the level-three mesh with \(11{,}776\) triangles.}
 \label{fig:num-trefoil-fields}
\end{figure}

\FloatBarrier

\subsection{A multiply connected plate with mixed boundary geometry}
\label{subsec:numerical-wrench}

The outer boundary of the third domain consists of two unequal circular
arcs joined by their common tangent segments. The circle centers are
\((0,1.82)\) and \((0,-1.58)\), with radii \(0.78\) and \(1.12\). The
smaller end contains a concentric circular hole of radius \(0.34\). The
larger end contains a regular hexagonal hole of circumradius \(0.56\),
rotated through \(17^\circ\). We take
\[
 u(x,y)=0.5+0.12x-0.08y+0.02\eta(y+0.78)
 \left(x^4+\frac35x^2y^2+\frac7{10}y^4+\frac12xy^3\right),
\]
where \(\eta(t)=\exp(-1/t^2)\) for \(t>0\) and \(\eta(t)=0\) for
\(t\leq0\). We set \(f=\Delta^2u\) and prescribe the boundary data
induced by \(u\) and \(\bm M=D^2u\). Near the hexagonal hole, \(u\) is
affine and hence \(M_{nn}=0\).
The initial mesh has \(649\) triangles. Four uniform refinements give
\(166{,}144\) triangles and \(h_{\max}=2.1298\times10^{-2}\).

The results in Table~\ref{tab:num-wrench} approach order \(3/2\) for the
moment and the postprocessed broken Hessian, and order two for both
displacement errors. Thus, the same correction applies without change to
multiple boundary components and to boundaries combining curved pieces,
straight pieces, and geometric corners. Figure~\ref{fig:num-wrench-fields}
shows the mesh and the computed bending moment.

\begin{table}[!htbp]
\centering
\caption{Multiply connected long plate: errors and estimated orders of convergence.}
\label{tab:num-wrench}
\scriptsize\setlength{\tabcolsep}{3.0pt}
\renewcommand{\arraystretch}{1.06}
\begin{tabular}{@{}*{9}{c}@{}}
\toprule
$h$ & \multicolumn{2}{c}{$\|u-u_h\|_0$} & \multicolumn{2}{c}{$\|u-u_h^\star\|_0$} & \multicolumn{2}{c}{$\|D_h^2(u-u_h^\star)\|_0$} & \multicolumn{2}{c}{$\|\bm M-\bm M_h\|_0$}\\
\cmidrule(lr){2-3}\cmidrule(lr){4-5}\cmidrule(lr){6-7}\cmidrule(lr){8-9}
0.3408 & $4.37\mathrm{e}{-4}$ & \multicolumn{1}{c}{--} & $8.26\mathrm{e}{-5}$ & \multicolumn{1}{c}{--} & $1.05\mathrm{e}{-2}$ & \multicolumn{1}{c}{--} & $1.33\mathrm{e}{-2}$ & \multicolumn{1}{c}{--}\\
0.1704 & $1.08\mathrm{e}{-4}$ & 2.01 & $8.80\mathrm{e}{-6}$ & 3.23 & $3.68\mathrm{e}{-3}$ & 1.52 & $4.71\mathrm{e}{-3}$ & 1.50\\
0.0852 & $2.70\mathrm{e}{-5}$ & 2.00 & $1.11\mathrm{e}{-6}$ & 2.99 & $1.29\mathrm{e}{-3}$ & 1.51 & $1.66\mathrm{e}{-3}$ & 1.50\\
0.0426 & $6.75\mathrm{e}{-6}$ & 2.00 & $2.51\mathrm{e}{-7}$ & 2.14 & $4.53\mathrm{e}{-4}$ & 1.51 & $5.86\mathrm{e}{-4}$ & 1.50\\
0.0213 & $1.69\mathrm{e}{-6}$ & 2.00 & $7.09\mathrm{e}{-8}$ & 1.82 & $1.60\mathrm{e}{-4}$ & 1.50 & $2.07\mathrm{e}{-4}$ & 1.50\\
\bottomrule
\end{tabular}
\end{table}
 
\begin{figure}[!htbp]
 \centering
 \includegraphics[width=0.98\textwidth]{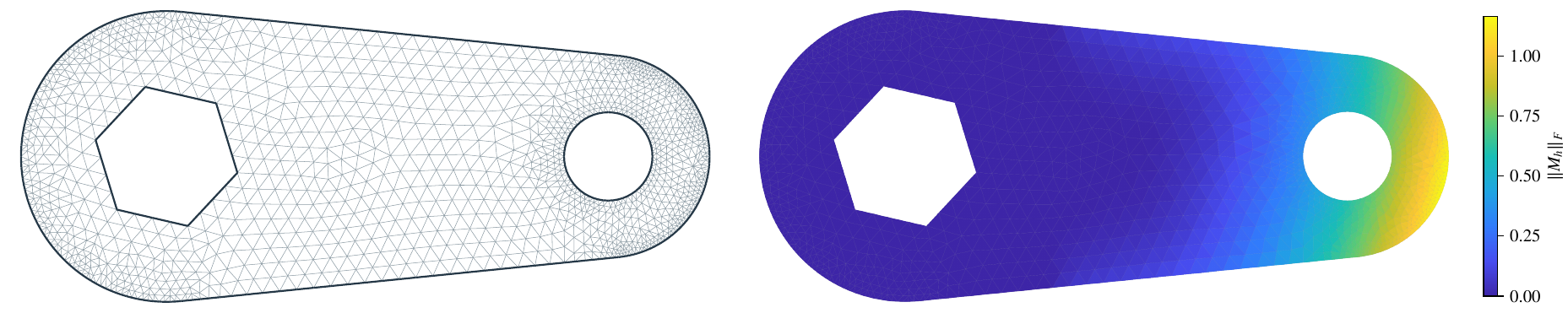}
 \caption{Multiply connected long plate: the level-one mesh (left) and
 the discrete bending-moment magnitude \(|\bm M_h|\), with mesh
 lines, on the same grid (right).}
 \label{fig:num-wrench-fields}
\end{figure}

\FloatBarrier

Although postprocessing does not change the observed second-order rate of
the displacement \(L^2\)-error, it substantially reduces its magnitude in
all three examples. On the finest meshes, the reduction factors are
approximately \(2.1\), \(12.9\), and \(23.8\), respectively.
 \leavevmode\par

\appendix
\section{Uniform discrete boundary lifting}
\label{app:boundary-lifting}
We prove the boundary lifting used in Lemma~\ref{lem:distributed-normal-lifting} (normal--normal constant-mode correction).

\begin{lemma}[uniform discrete $H^1$ boundary lifting]
\label{lem:discrete-trace-lifting}
Under Assumption~\ref{ass:polygonal-boundary}, 
for every $p_h^\partial\in C^0(\Gamma_h)$ satisfying $p_h^\partial|_e\in \mathbb{P}_1(e)$ for all $e\in\mathcal E_h^\partial$, 
there exists $p_h\in S_h^1$, depending linearly on $p_h^\partial$,
such that
\begin{equation}
\label{eq:discrete-trace-lifting}
  p_h|_{\Gamma_h}=p_h^\partial,
  \qquad
  \|p_h\|_{H^1(\Omega_h)}
  \lesssim
  \|p_h^\partial\|_{H^1(\Gamma_h)}.
\end{equation}
If $p_h^\partial$ has zero mean on every connected
component of $\Gamma_h$, then
\begin{equation}
\label{eq:mean-zero-discrete-trace-lifting}
  \|p_h\|_{H^1(\Omega_h)}
  \lesssim
  \|\partial_t p_h^\partial\|_{L^2(\Gamma_h)}.
\end{equation}
The hidden constants are uniform for $0<h\le h_0$, with $h_0>0$
sufficiently small.
\end{lemma}
\begin{proof}
Fix a finite Lipschitz cover of \(\Gamma\) and a smooth partition of
unity \(\{\chi_j\}\) on a neighborhood of \(\Gamma\), with supports
compactly contained in the corresponding charts. By the boundary interpolation and Hausdorff convergence
in Assumption~\ref{ass:polygonal-boundary}, the same charts
cover \(\Gamma_h\) and \(\sum_j\chi_j=1\) on \(\Gamma_h\) for all
sufficiently small \(h\).

We now consider one chart and suppress its index. In its fixed rigid
coordinates \((s,r)\), write the physical and polygonal boundaries as
\(r=\gamma(s)\) and \(r=\gamma_h(s)\), respectively, with the domains
lying above the graphs. The piecewise affine interpolant \(\gamma_h\)
inherits the Lipschitz bound \(\|\gamma_h'\|_{L^\infty}\le L\) of
\(\gamma\) and converges uniformly to it. The interior margin of the
fixed chart therefore provides a width \(\delta>0\), independent of
\(h\), such that \((s,\gamma_h(s)+\xi)\in\Omega_h\) for
\(0<\xi<\delta\) over the support of \(\chi|_{\Gamma_h}\).

Set \(g_h(s):=(\chi p_h^\partial)(s,\gamma_h(s))\), \(s\in I\), where
\(I\) is the coordinate interval of the chart. Choose a fixed smooth
cutoff \(\eta\) with \(\eta(0)=1\) and vanishing near
one, and define the local lifting by
\[
  w(s,\gamma_h(s)+\xi):=\eta(\xi/\delta)g_h(s),
  \qquad 0<\xi<\delta.
\]
The support conditions permit extension by zero to the rest of
\(\Omega_h\). Direct differentiation gives
\[
  \|w\|_{H^1(\Omega_h)}^2
  \le C(L)\Bigl[
    \delta\|g_h'\|_{L^2(I)}^2
    +(\delta+\delta^{-1})\|g_h\|_{L^2(I)}^2
  \Bigr].
\]
The comparison of coordinate length with arclength depends only on
\(L\), while the partition functions and their derivatives are fixed.
Summing these local liftings therefore yields \(w_h\in H^1(\Omega_h)\)
with trace \(p_h^\partial\) and
\[
  \|w_h\|_{H^1(\Omega_h)}
  \le C_\Omega\|p_h^\partial\|_{H^1(\Gamma_h)}.
\]
Here \(C_\Omega\) depends on the fixed cover, Lipschitz
bounds, and interior widths.

Set \(p_h:=I_h^\partial w_h\), where \(I_h^\partial\) is the
Scott--Zhang projection with boundary-edge averaging at boundary
vertices \cite{ScottZhang1990}. Its boundary preservation and uniform
\(H^1\)-stability give \eqref{eq:discrete-trace-lifting}.
If the trace has zero mean on each boundary component, the periodic
Poincar\'e inequality and
\(|\Gamma_h^{(i)}|\le |\Gamma^{(i)}|\) give
\(\|p_h^\partial\|_{H^1(\Gamma_h)}
\lesssim\|\partial_t p_h^\partial\|_{L^2(\Gamma_h)}\), proving
\eqref{eq:mean-zero-discrete-trace-lifting}. All steps in the construction
are linear in \(p_h^\partial\), which completes the proof.
\end{proof}


\begin{thebibliography}{10}

\bibitem{ArnoldWalker2020}
{\sc D.~N. Arnold and S.~W. Walker}, {\em The {Hellan--Herrmann--Johnson}
  method with curved elements}, SIAM Journal on Numerical Analysis, 58 (2020),
  pp.~2829--2855, \url{https://doi.org/10.1137/19M1288723}.

\bibitem{Babuska1963}
{\sc I.~Babu\v{s}ka}, {\em The theory of small changes in the domain of
  existence in the theory of partial differential equations and its
  applications}, in Differential Equations and Their Applications: Proceedings
  of the Conference Held in Prague in September 1962, Prague, 1963, Publishing
  House of the Czechoslovak Academy of Sciences, pp.~13--26.

\bibitem{BabuskaPitkaranta1990}
{\sc I.~Babu\v{s}ka and J.~Pitk\"aranta}, {\em The plate paradox for hard and
  soft simple support}, SIAM Journal on Mathematical Analysis, 21 (1990),
  pp.~551--576, \url{https://doi.org/10.1137/0521030}.

\bibitem{BartelsTscherner2025}
{\sc S.~Bartels and P.~Tscherner}, {\em Necessary and sufficient conditions for
  avoiding {Babu\v{s}ka}'s paradox on simplicial meshes}, IMA Journal of
  Numerical Analysis, 45 (2025), pp.~1300--1319,
  \url{https://doi.org/10.1093/imanum/drae050}.

\bibitem{BoffiBrezziFortin2013}
{\sc D.~Boffi, F.~Brezzi, and M.~Fortin}, {\em Mixed Finite Element Methods and
  Applications}, vol.~44 of Springer Series in Computational Mathematics,
  Springer, Heidelberg, 2013, \url{https://doi.org/10.1007/978-3-642-36519-5}.

\bibitem{BrennerNeilanSung2013}
{\sc S.~C. Brenner, M.~Neilan, and L.-Y. Sung}, {\em Isoparametric {$C^0$}
  interior penalty methods for plate bending problems on smooth domains},
  Calcolo, 50 (2013), pp.~35--67,
  \url{https://doi.org/10.1007/s10092-012-0057-1}.

\bibitem{ChechkinLukkassenMeidell2008}
{\sc G.~A. Chechkin, D.~Lukkassen, and A.~Meidell}, {\em On the
  {Sapondzhyan--Babu\v{s}ka} paradox}, Applicable Analysis, 87 (2008),
  pp.~1443--1460.

\bibitem{ChenHuangDivDiv2020}
{\sc L.~Chen and X.~Huang}, {\em Finite elements for divdiv-conforming
  symmetric tensors}, 2020, \url{https://arxiv.org/abs/2005.01271}.

\bibitem{ChenHuangArbitrary2022}
{\sc L.~Chen and X.~Huang}, {\em Finite elements for div- and divdiv-conforming
  symmetric tensors in arbitrary dimension}, SIAM Journal on Numerical
  Analysis, 60 (2022), pp.~1932--1961,
  \url{https://doi.org/10.1137/21M1433708}.

\bibitem{ChenHuang3D2022}
{\sc L.~Chen and X.~Huang}, {\em Finite elements for div div conforming
  symmetric tensors in three dimensions}, Mathematics of Computation, 91
  (2022), pp.~1107--1142, \url{https://doi.org/10.1090/mcom/3700}.

\bibitem{ChenHuangHybrid2025}
{\sc L.~Chen and X.~Huang}, {\em A new div-div-conforming symmetric tensor
  finite element space with applications to the biharmonic equation},
  Mathematics of Computation, 94 (2025), pp.~33--72,
  \url{https://doi.org/10.1090/mcom/3957}.

\bibitem{Davini2002}
{\sc C.~Davini}, {\em {$\Gamma$}-convergence of external approximations in
  boundary value problems involving the bi-laplacian}, Journal of Computational
  and Applied Mathematics, 140 (2002), pp.~185--208.

\bibitem{Davini2003}
{\sc C.~Davini}, {\em Gaussian curvature and {Babu\v{s}ka}'s paradox in the
  theory of plates}, in Rational Continua, Classical and New, Springer Italia,
  Milan, 2003, pp.~67--87, \url{https://doi.org/10.1007/978-88-470-2231-7_6}.

\bibitem{DaviniPitacco2000}
{\sc C.~Davini and I.~Pitacco}, {\em An unconstrained mixed method for the
  biharmonic problem}, SIAM Journal on Numerical Analysis, 38 (2000),
  pp.~820--836, \url{https://doi.org/10.1137/S0036142998347833}.

\bibitem{DeCosterNicaiseSweers2019}
{\sc C.~De~Coster, S.~Nicaise, and G.~Sweers}, {\em Comparing variational
  methods for the hinged {Kirchhoff} plate with corners}, Mathematische
  Nachrichten, 292 (2019), pp.~2574--2601.

\bibitem{FuehrerHeuer2025}
{\sc T.~F{\"u}hrer and N.~Heuer}, {\em Mixed finite elements for
  {Kirchhoff--Love} plate bending}, Mathematics of Computation, 94 (2025),
  pp.~1065--1099, \url{https://doi.org/10.1090/mcom/3995}.

\bibitem{Grisvard1985}
{\sc P.~Grisvard}, {\em Elliptic Problems in Nonsmooth Domains}, vol.~24 of
  Monographs and Studies in Mathematics, Pitman, Boston, 1985.

\bibitem{HuLiangMaZhang2024}
{\sc J.~Hu, Y.~Liang, R.~Ma, and M.~Zhang}, {\em A family of conforming finite
  element divdiv complexes on cuboid meshes}, Numerische Mathematik, 156
  (2024), pp.~1603--1638, \url{https://doi.org/10.1007/s00211-024-01418-7}.

\bibitem{HuMaZhang2021}
{\sc J.~Hu, R.~Ma, and M.~Zhang}, {\em A family of mixed finite elements for
  the biharmonic equations on triangular and tetrahedral grids}, Science China
  Mathematics, 64 (2021), pp.~2793--2816,
  \url{https://doi.org/10.1007/s11425-020-1883-9}.

\bibitem{MazyaNazarov1986}
{\sc V.~G. Maz'ya and S.~A. Nazarov}, {\em Paradoxes of the passage to the
  limit in solutions of boundary value problems for the approximation of smooth
  domains by polygons}, Izvestiya Akademii Nauk SSSR. Seriya Matematicheskaya,
  50 (1986), pp.~1156--1177, 1343.

\bibitem{NazarovSweersStilyanou2011}
{\sc S.~A. Nazarov, G.~Sweers, and A.~Stilyanou}, {\em On paradoxes in problems
  of the bending of polygonal plates with ``hinge-supported'' edges}, Doklady
  Akademii Nauk, 439 (2011), pp.~476--480.

\bibitem{RannacherParadox1979}
{\sc R.~Rannacher}, {\em Finite element approximation of simply supported
  plates and the {Babu\v{s}ka} paradox}, Zeitschrift f\"ur Angewandte
  Mathematik und Mechanik, 59 (1979), pp.~T73--T76.

\bibitem{RannacherMixed1979}
{\sc R.~Rannacher}, {\em On nonconforming and mixed finite element methods for
  plate bending problems: The linear case}, RAIRO Analyse Num\'erique, 13
  (1979), pp.~369--387, \url{https://doi.org/10.1051/m2an/1979130403691}.

\bibitem{Scott1977}
{\sc L.~R. Scott}, {\em A survey of displacement methods for the plate bending
  problem}, in Formulations and Computational Algorithms in Finite Element
  Analysis, The MIT Press, Cambridge, MA, 1977, pp.~855--876.

\bibitem{ScottZhang1990}
{\sc L.~R. Scott and S.~Zhang}, {\em Finite element interpolation of nonsmooth
  functions satisfying boundary conditions}, Mathematics of Computation, 54
  (1990), pp.~483--493, \url{https://doi.org/10.2307/2008497}.

\bibitem{UtkuCarey1983}
{\sc M.~Utku and G.~F. Carey}, {\em Penalty resolution of the {Babu\v{s}ka}
  circle paradox}, Computer Methods in Applied Mechanics and Engineering, 41
  (1983), pp.~11--28, \url{https://doi.org/10.1016/0045-7825(83)90050-6}.

\end{thebibliography}
\end{document}